\documentclass[11pt,leqno]{amsart}
\usepackage{enumerate}
\usepackage{amssymb,amsmath,amsthm,amsfonts}
\usepackage {latexsym}
\usepackage{bbm}
\usepackage{color}

\allowdisplaybreaks
\newcommand{\nc}{\newcommand}
\nc{\les}{\lesssim}
\nc{\nit}{\noindent}
\nc{\nn}{\nonumber}
\nc{\D}{\partial}
\nc{\diff}[2]{\frac{d #1}{d #2}}
\nc{\diffn}[3]{\frac{d^{#3} #1}{d {#2}^{#3}}}
\nc{\pdiff}[2]{\frac{\partial #1}{\partial #2}}
\nc{\pdiffn}[3]{\frac{\partial^{#3} #1}{\partial{#2}^{#3}}}
\nc{\abs}[1] {\lvert #1 \rvert}
\nc{\cAc}{{\cal A}_c}
\nc{\cE}{{\cal E}}
\nc{\cF}{{\cal F}}
\nc{\cP}{{\cal P}}
\nc{\cV}{{\cal V}}
\nc{\cQ}{{\cal Q}}
\nc{\cGin}{{\cal G}_{\rm in}}
\nc{\cGout}{{\cal G}_{\rm out}}
\nc{\cO}{{\cal O}}
\nc{\Lav}{{\cal L}_{\rm av}}
\nc{\cL}{{\cal L}}
\nc{\cB}{{\cal B}}
\nc{\cZ}{{\cal Z}}
\nc{\cR}{{\cal R}}
\nc{\cT}{{\cal T}}
\nc{\cY}{{\cal Y}}
\nc{\cX}{{\cal X}}
\nc{\cXT}{{{\cal X}(T)}}
\nc{\cBT}{{{\cal B}(T)}}
\nc{\vD}{{\vec \mathcal{D}}}
\nc{\efield}{\mathcal{E}}
\nc{\vE}{{\vec \efield}}
\nc{\vB}{{\vec \mathcal{B}}}
\nc{\vH}{{\vec \mathcal{H}}}
\nc{\ty}{{\tilde y}}
\nc{\tu}{{\tilde u}}
\nc{\tV}{{\tilde V}}
\nc{\Pc}{{\bf P_c}}
\nc{\bx}{{\bf x}}
\nc{\bX}{{\bf X}}
\nc{\bXYZ}{{\bf XYZ}}
\nc{\bY}{{\bf Y}}
\nc{\bF}{{\bf F}}
\nc{\bS}{{\bf S}}
\nc{\dV}{{\delta V}}
\nc{\dE}{{\delta E}}
\nc{\TT}{{\Theta}}
\nc{\dPsi}{{\delta\Psi}}
\nc{\order}{{\cal O}}
\nc{\Rout}{R_{\rm out}}
\nc{\eplus}{e_+}
\nc{\eminus}{e_-}
\nc{\epm}{e_\pm}
\nc{\eps}{\varepsilon}
\nc{\vnabla}{{\vec\nabla}}
\nc{\G}{\Gamma}
\nc{\w}{\omega}
\nc{\mh}{h}
\nc{\mg}{g}
\nc{\vphi}{\varphi}
\nc{\tlambda}{\tilde\lambda}
\nc{\be}{\begin{equation}}
\nc{\ee}{\end{equation}}
\nc{\ba}{\begin{eqnarray}}
\nc{\ea}{\end{eqnarray}}

\nc{\g}{\gamma}
\nc{\ol}{\overline}

\newtheorem{theorem}{Theorem}[section]
\newtheorem{lemma}[theorem]{Lemma}
\newtheorem{prop}[theorem]{Proposition}
\newtheorem{corollary}[theorem]{Corollary}

\newtheorem{rmk}[theorem]{Remark}

\nc{\pT}{\partial_T}
\nc{\pz}{\partial_z}
\nc{\pt}{\partial_t}
\nc{\la}{\langle}
\nc{\ra}{\rangle}
\nc{\infint}{\int_{-\infty}^{\infty}}
\nc{\halfwidth}{6.5cm}
\nc{\figwidth}{10cm}
\newcommand{\f}{\frac}

\nc{\nlayers}{L} \nc{\nsectors}{M}
\nc{\indicator}{\mathbf{1}}
\nc{\Rhole}{R_{\rm hole}}
\nc{\Rring}{R_{\rm ring}}
\nc{\neff}{n_{\rm eff}}
\nc{\Frem}{F_{\rm rem}}
\nc{\R}{\mathbb R}
\nc{\Z}{\mathbb Z}
\nc{\DD}{\Delta}
\nc{\cD}{\mathcal D}
\nc{\lnorm}{\left\|}
\nc{\rnorm}{\right\|}
\nc{\rnormp}{\right\|_{\ell^{p,\eps}}}
\nc{\rar}{\rightarrow}
\date{\today}
\begin{document}

\begin{abstract}

Let $H=-\Delta+V$ be a Schr\"odinger operator on $L^2(\R^n)$ with real-valued potential $V$
for $n > 4$ and let $H_0=-\Delta$.  If $V$ decays
sufficiently, the wave operators
$W_{\pm}=s-\lim_{t\to \pm\infty} e^{itH}e^{-itH_0}$ are
known to be bounded on $L^p(\R^n)$ for all $1\leq p\leq \infty$ if zero is not an eigenvalue,
and on $1<p<\frac{n}{2}$ if zero is an eigenvalue.
We show that these wave operators are also bounded on $L^1(\R^n)$ by direct examination of the integral kernel
of the leading term.
Furthermore,
if $\int_{\R^n} V(x) \phi(x) \, dx=0$
for all eigenfunctions $\phi$, then the wave operators
are $L^p$ bounded for $1\leq p<n$.  If, in addition
$\int_{\R^n} xV(x) \phi(x) \, dx=0$, then the wave operators
are bounded for $1 \leq p<\infty$.

\end{abstract}

\title[$L^p$ boundedness of wave operators]{\textit{The $L^p$ boundedness of wave operators for Schr\"odinger Operators with threshold singularities }}

\author[M.~J. Goldberg, W.~R. Green]{Michael Goldberg and William~R. Green}

\address{Department of Mathematics \\
University of Cincinnati\\
Cincinnati, OH 45221-0025}
\email{Michael.Goldberg@uc.edu}
\address{Department of Mathematics\\
Rose-Hulman Institute of Technology \\
Terre Haute, IN 47803 U.S.A.}
\email{green@rose-hulman.edu}


\thanks{The first author 
was  supported  by Simons  Foundation grant \#281057
during the preparation of this work}

\date{October 6, 2015}

\maketitle

\section{Introduction}

Let $H=-\Delta+V$ be a Schr\"odinger operator
with potential $V$ and
$H_0=-\Delta$.  If $V$ is real-valued and satisfies
$|V(x)|\les \la x\ra^{-2-}$, then it is well known that
the spectrum of $H$ is the absolutely continuous
spectrum on $[0,\infty)$ and a finite collection of
non-positive eigenvalues, \cite{RS1}.
The wave operators are defined by the strong limits on $L^2(\R^n)$
\begin{align}
	W_{\pm}=\lim_{t\to\pm \infty} e^{itH}e^{-itH_0}.
\end{align}
Such limits are known to exist and are asymptotically
complete  for a wide class of
potentials $V$.  That is, the image of $W_{\pm}$ is equal to the
absolutely continuous subspace of $L^2(\R^n)$ associated
to the Schr\"odinger operator $H$.  Furthermore, one has
the identities
\begin{align}
	W_\pm^* W_\pm=I, \qquad W_\pm W_\pm^*=P_{ac}(H),
\end{align}
with $P_{ac}(H)$ the projection onto the absolutely continous
spectral subspace associated with the Schr\"odinger operator $H$.

We say that zero energy is regular if there are no zero
energy eigenvalues or resonances.
There is a zero energy eigenvalue if there is a
solution to $H\psi =0$ with $\psi\in L^2(\R^n)$, and a
resonance if $\psi\notin L^2(\R^n)$ is in an appropriate space which
depends on the dimension.  We note
that resonances only occur in dimensions $n\leq 4$. 
There is a long history of results on the existence and
boundedness of the wave operators.  We note that Yajima
has established $L^p$ and $W^{k,p}$ boundedness of the
wave operators for the full range of $1\leq p\leq \infty$ in \cite{YajWkp1,YajWkp2,YajWkp3} in all
dimensions $n\geq 3$, provided that zero energy is regular
under varying assumptions on the potential $V$.  The sharpest
result in $n=3$ was obtained by Beceanu in \cite{Bec}.

If zero is not regular, in general the range of $p$ on which the wave operators are bounded shrinks.
Yajima (for $n$ odd), and Yajima and Finco (for $n$ even) proved in \cite{Yaj,FY} that for each
$n>4$, the wave operators are bounded on $L^p(\R^n)$ when
$\frac{n}{n-2}<p<\frac{n}{2}$, and on $\frac{3}{2}<p<3$ in
$n=3$ if zero is not regular.  In \cite{JY4} Jensen and
Yajima showed that
the wave operators are bounded if $\frac{4}{3}<p<4$
when $n=4$ when there is an eigenvalue but no resonance at
zero.  D'Ancona and Fanelli in \cite{DF} show that the wave operators are bounded on $L^p(\R)$ for $1<p<\infty$ in the
case of a zero energy resonance, which had roots in the work of Weder, \cite{Wed}.  To the best of the
authors' knowledge, there are no results in the literature
when zero is not regular and $n=2$.  
Very
recently Yajima,  in \cite{YajNew}, reduced
the lower bound on $p$ to $1<p<\f n2$ for dimensions $n>4$ when there is a
zero energy eigenvalue.
We extend this result to include the $p=1$ endpoint.

One important property of the wave operators is the
intertwining identity,
$$
	f(H)P_{ac}=W_{\pm}f(-\Delta)W_{\pm}^*,
$$
which is valid for Borel functions $f$.  This allows one
to deduce properties of the operator $f(H)$ from the
much simpler operator $f(-\Delta)$, provided one has
control on mapping properties of the wave operators
$W_\pm$ and $W_\pm^*$.
In dimensions $n\geq 5$, boundedness of the 
wave operators on for the range of $p$ proven in \cite{Yaj,FY} imply the dispersive estimates
$$
	\|e^{itH}P_{ac}(H)\|_{L^p\to L^{p\prime}}\les
	|t|^{-\f n2+\frac{n}{p}}.
$$
Here $p^\prime$ is the conjugate exponent satisfying
$\frac{1}{p}+\frac{1}{p^\prime}=1$.  
In this way, one can
use the $L^p$ boundedness of the wave operators to 
deduce dispersive estimates for the Schr\"odinger evolution.  There has been much work on dispersive estimates
for the Schr\"odinger evolution with zero energy obstructions in recent years by Erdo\smash{\u{g}}an, Schlag and the authors in various combinations, see \cite{ES2,goldE,EG,EGG,GGodd,GGeven} in which $L^1(\R^n)\to L^\infty(\R^n)$ were studied for all $n>1$.  This work 
has roots in previous work of \cite{JSS} and 
\cite{Jen,Mur} in which the dispersive estimates were studied as operators on weighted
$L^2(\R^n)$ spaces.

The range of $p$ proven in \cite{Yaj} allows
one to deduce a decay rate of size $|t|^{-\f n2+2+}$.
This paper is motivated by the recent work of
the authors, \cite{GGodd,GGeven}, in which dispersive estimates
with a decay rate of $|t|^{2-\f n2}$ were proven in the
case of an eigenvalue at zero energy, and faster decay if the
zero energy eigenspace satisfies certain cancellation conditions.
Let $P_e$ be the projection onto the zero energy eigenspace, 
and write $P_eV1=0$ if $\int_{\R^n} V(x) \phi(x) \, dx=0$
for each eigenfunction $\phi$,
and $P_eVx=0$ if $\int_{\R^n} xV(x) \phi(x) \, dx=0$.  
Considering the linear Schr\"odinger evolution as an operator from $L^1(\R^n)$
to $L^\infty(\R^n)$, time decay of size
$|t|^{1-\f n2}$ is observed if $P_eV1=0$ and if in addition
$P_eVx=0$, the decay rate improves to $|t|^{-\f n2} $.

Time decay of these orders would be consistent with $L^p$
boundedness of the wave operators over the range 
$1 \leq p \leq n$ if $P_eV1=0$.  In the case that
$P_eVx=0$ as well, the time-decay is identical to what occurs
in the free case, so it is conceivable for the range to extend to
$1 \leq p \leq \infty$.  Our main result confirms this
to be the case for all $p$ except the upper endpoints.  During the review period for this article Yajima additionally showed that the orthogonality conditions are also necessary for the extended range of $L^p$ boundedness, \cite{YajNew2}.
\begin{theorem}\label{thm:main}

	For each $n>4$, let
	$n_*=\frac{n-1}{n-2}$.  Assume that
	$|V(x)|\les \la x\ra^{-\beta}$ for 
	$\beta>n+3$, or $\beta > 16$ if $n=6$, and
	\begin{align}\label{Vfour}
		\mathcal F\big(\la \cdot \ra^{2\sigma} V \big)
		\in L^{n_*}(\R^n) \textrm{ for some } \sigma>\frac{1}{n_*}.
	\end{align}
	\begin{enumerate}[i)]
		\item \label{reg result}
		The wave operators extend to bounded operators on $L^p(\R^n)$
		for all $1 \leq p < \frac{n}{2}$.
		
		\item \label{PV1 result} If  
		$\int_{\R^n} V(x) \phi(x) \, dx=0$
		for all zero-energy
		eigenfunctions $\phi$, then the wave operators
		extend to bounded operators on $L^p(\R^n)$ for all
		$1 \leq p<n$.
		
		\item \label{Pvx result} If 
		$\int_{\R^n} V(x) \phi(x) \, dx=0$ and 
		$\int_{\R^n} xV(x) \phi(x) \, dx=0$
		for all zero-energy
		eigenfunctions $\phi$, then the wave operators
		extend to bounded operators on $L^p(\R^n)$ for all
		$1 \leq p<\infty$.		
		
	\end{enumerate}
	
\end{theorem}

Except when $n=6$, one can extend the arguments presented here to show that the endpoint case $p=\infty$
holds if one has the additional cancellation $\int_{\R^n} x^2 V(x) \phi(x) \, dx=0$ and slightly more
decay on the potential, see Remark~\ref{PVx2 rmk} below.  
The $L^p$ bounds can be extended to boundedness as an operator on $W^{k,p}(\R^n)$ for the
same range of $1<p<\infty$ with $0\leq k\leq 2$ by a standard
argument that shows an equivalence between the norms
$\|u\|_{W^{k,p}}$ and $\|(-\Delta +c^2)u\|_{L^p}$ for a
sufficiently large constant $c$, see \cite{YajNew}.

We prove the results for $W=W_-$, 
the proof for $W_+$ is similar.
Following the approach of Yajima in \cite{Yaj},
the starting point is the  stationary representation of the wave operator
\begin{equation}
\begin{aligned}\label{stat rep}
	Wu &=u-\frac{1}{\pi i} \int_0^\infty \lambda R_V^+(\lambda^2)V
	[R_0^+(\lambda^2)-R_0^-(\lambda^2)]\, u\, d\lambda \\
	&= u-\frac{1}{\pi i} \int_0^\infty \lambda \big[R_0^+(\lambda^2) -R_0^+(\lambda^2)VR_V^+(\lambda^2)\big]V
	[R_0^+(\lambda^2)-R_0^-(\lambda^2)]\, u\, d\lambda
\end{aligned}
\end{equation}
where $R_0^\pm(\lambda^2) := \lim\limits_{\eps \to 0^+} (H_0 - (\lambda \pm i\eps)^2)^{-1}$
and $R_V^+(\lambda^2) := \lim\limits_{\eps \to 0^+} (H - (\lambda +i\eps)^2)^{-1}$ 
are the free and perturbed resolvents, respectively. 
These operators are known to be well-defined on polynomially
weighted $L^2(\R^n)$ spaces due to the limiting absorption
principle, \cite{agmon}.  In dimensions $n>2$, the free
resolvent operators $R_0^\pm(\lambda^2)$ are bounded as $\lambda \to 0$, as
are the perturbed resolvents $R_V^\pm(\lambda^2)$ if zero is regular.  When zero is not regular,
the perturbed resolvent
becomes singular as $\lambda\to 0$.  This singular behavior
shrinks the range of $p$ on which
the wave operators are $L^p(\R^n)$ bounded.
 
The last equality in \eqref{stat rep} follows from the standard 
resolvent identity $R_V^+(\lambda^2)=R_0^+(\lambda^2)-R_0^+(\lambda^2)VR_V^+(\lambda^2)$.
One can then split $W$ into high and low energy parts,
$W=W\Phi^2(H_0)+W \Psi^2(H_0)$ with $\Phi, \Psi \in C_0^\infty(\R)$ smooth cut-off functions that satisfy
$\Phi^2(\lambda)+\Psi^2(\lambda)=1$ with $\Phi(\lambda^2)=1$ for $|\lambda|\leq \lambda_0/2$ and
$\Phi(\lambda^2)=0$ for $|\lambda|\geq \lambda_0$ for a
suitable constant $0<\lambda_0\ll 1$.  This allows us to 
write $W=W_<+W_>$, with $W_<$ the `low energy' portion
of the wave operator and $W_>$ the `high energy' portion.
Taking advantage of the intertwining property,
one can express $W_>=\Psi(H)W\Psi(H_0)$ and
$W_<=\Phi(H)W\Phi(H_0)$.

The weighted Fourier bound on the
potential, \eqref{Vfour}, can be interpreted as requiring a certain amount
of smoothness on the potential $V$.  In light of 
the counterexample to dispersive estimates in \cite{GV} and the work in \cite{EG1},
it seems possible that one may be able to require less 
smoothness on the
potential, we do not pursue that issue here.

In \cite{Yaj}, it was shown  that
$W_>$ is bounded in $L^p(\R^n)$ for the full range of $1\leq p\leq \infty$ provided $|V(x)|\les \la x\ra^{-n-2-}$ and
\eqref{Vfour} holds.  The high
energy portion is unaffected by zero energy eigenvalues.

The fact that $n \geq 5$ allows for greater uniformity in the treatment of $W_<$, as there are no
special considerations related to the distinction between resonances
and eigenvalues at zero.  There are, however, significant differences in the low-energy
expansion of the resolvent depending on whether $n$ is even or odd,
with the even dimensions presenting more technical challenges due to some
logarithmic behavior near zero.  The low energy analysis becomes
progressively more idiosyncratic for small $n$, requiring additional
arguments here for $n = 5, 6, 8, 10$ in particular. 

Some results are also known when $n=4$ in the case where zero is an eigenvalue
but not a resonance.  The operator $W_<$ is shown in~\cite{JY4} to be bounded
on $L^p(\R^4)$ for $\frac43 < p < 4$, and was recently extended the range to
$1 \leq p < 4$ by the authors in~\cite{GGwaveop}.
Questions about the $L^p$ boundedness of the wave operators remain open if there
is a resonance in four dimensions, or any kind of zero energy 
obstruction in two dimensions.

The next section sketches an argument that controls the leading order expression
for $W_<$ when there is a zero energy eigenvalue.  We examine the integral kernel
of this operator in order to determine the range of exponents $p$ for which it
is bounded on $L^p(\R^n)$.  The argument relies on several important integral
estimates whose proofs are provided later in Section~\ref{sec:appA}.
In Section~\ref{sec:canc} we show that the argument is direct and flexible enough to be modified to
take advantage of the additional cancellation in the event that $P_eV1 = 0$,
and more so if $P_eVx = 0$ as well.  The summary proof of Theorem~\ref{thm:main}
is given immediately afterward.  Section~\ref{sec:appA} contains a full proof of
the key integral estimates.  Finally in section~\ref{sec:8,10} we address some
additional modifications that are necessary when $n = 6, 8, 10$ in order to control
terms of $W_<$ which are not leading order but nevertheless require 
further scrutiny.

\section{No cancellation}

When there is a zero energy eigenvalue, the perturbed resolvent
$R_V^+(\lambda^2)$ in~\eqref{stat rep} has a pole of order two
whose residue is the finite-rank projection $P_e$ onto the eigenspace.
The leading term in a low energy expansion for $W_<$ is therefore
given by the operator
$$
	W_{s,2}=\frac{1}{\pi i} \int_0^\infty R_0^+(\lambda^2)
	VP_eV(R_0^+(\lambda^2)-R_0^-(\lambda^2))
	\tilde \Phi(\lambda)\lambda^{-1}\, d\lambda.
$$
In this section we obtain pointwise bounds on the integral kernel
of $W_{s,2}$ in order to determine the range of $p$ for which
it is $L^p$ bounded.  

One can show that
the remaining terms in the expansion of $W_<$ are
better behaved.  Thus the estimates on $W_{s,2}$
dictate the mapping properties of $W_<$ itself.
The exact form of the low energy expansion is
heavily dependent on whether $n$ is even or odd and is
discussed more fully in Section~\ref{sec:canc} below.
The presence or absence of threshold eigenvalues has little effect
on properties of the resolvent outside a small neighborhood of
$\lambda = 0$, so the estimates for $W_>$ are unchanged.

We first consider this operator under 
the assumption that there is a zero energy eigenvalue,
but no further cancellation.  That is, we do not assume
that $P_eV1=0$ or $P_eVx=0$.  
The kernel of $W_{s,2}$ is a sum of integrals of the form
\begin{align}\label{eqn:Ws22}
	K^{jk}(x,y) = \int_0^\infty \iint_{\R^{2n}} R_0^{+}(\lambda^2)(x,z)& V(z)\phi_{j}(z)
	V(w) \phi_{k}(w) \\  
	&(R_0^+-R_0^-)(\lambda^2)(w,y)
	\frac{\tilde\Phi(\lambda)}{\lambda} \, dwdz\, d\lambda \nn
\end{align}
where the functions $\{\phi_j\}_{j=1}^N$ form an orthonormal basis for the zero energy eigenspace, and
$\tilde \Phi(\lambda)\in C_c^\infty(\R)$ is such 
that $\tilde \Phi(\lambda)\Phi(\lambda^2)=\Phi(\lambda^2)$.  In \cite{Yaj,YajNew}, Yajima converts the integrals
to one-dimensional integrals 
and proves the desired $L^p$ bounds using the harmonic analysis tools of $A_p$ weights, maximal functions and
Hilbert transforms.  We approach the same problem by estimating
the integrals in $\R^n$ directly.   This allows us to recover
Yajima's result for $W_{s,2}$ while also obtaining
the $p=1$ endpoint.  The intermediate steps
can be modified to improve the range of $p$ if there is adequate
cancellation, which we show in Section ~\ref{sec:canc}.

For the remainder of the paper, we omit the subscripts 
on the zero-energy eigenfunctions as our calculations will be
satisfied for any such $\phi$.  Our main estimates are therefore
stated for an operator kernel $K(x,y)$ with the understanding that
each $K^{jk}(x,y)$ obeys the same bounds.  
We only utilize the
natural decay of $V(z)\phi(z)$ and (later in Section~\ref{sec:canc})
the cancellation hypotheses
in Theorem~\ref{thm:main} which hold for every $\phi$ in the
zero energy eigenspace.  
We first describe the natural decay of zero-energy
eigenfunctions.  
\begin{lemma}\label{lem:efn decay}

	If $|V(x)|\les \la x\ra^{-2-\epsilon}$ 
	for some $\epsilon>0$, and
	$\phi$ is a zero-energy eigenfunction, then
	$|\phi(x)|\les \la x\ra^{2-n}$.

\end{lemma}

\begin{proof}

	We note from Lemma~5.2 of \cite{GGodd}, 
	that any eigenfunction $\phi\in L^\infty(\R^n)$.
	We then rewrite $(-\Delta+V)\phi=0$ as
	$(I+(-\Delta)^{-1}V)\phi=0$.  It is well-known that
	$(-\Delta)^{-1}$ is an integral operator with
	integral kernel $c_n |x-y|^{2-n}$.  Thus, we have
	\begin{align*}
		|\phi(x)|=\bigg| c_n \int_{\R^n}\frac{V(y)\phi(y)}{|x-y|^{n-2}}\, dy
		\bigg| &\les \| \phi \|_{\infty} \int_{\R^n}
		\frac{\la y\ra^{-2-\epsilon}}{|x-y|^{n-2}} \, dy
		\les \la x\ra^{-\epsilon}.
	\end{align*}
	The last integral bound is easily proven, see for
	example Lemma~3.8 of \cite{GV}.  This estimate allows
	us to bootstrap, increasing the decay of $\phi$ at
	each step.  
	\begin{align*}
		|\phi(x)|=\bigg| c_n \int_{\R^n}\frac{V(y)\phi(y)}{|x-y|^{n-2}}\, dy
		\bigg| &\les \int_{\R^n}
		\frac{\la y\ra^{-2-2\epsilon}}{|x-y|^{n-2}} \, dy
		\les \la x\ra^{-2\epsilon}.
	\end{align*}		
	After $\frac{n-2}{\epsilon}$ iterations,
	one has
	\begin{align*}
		|\phi(x)|=\bigg| c_n \int_{\R^n}\frac{V(y)\phi(y)}{|x-y|^{n-2}}\, dy
		\bigg| &\les \int_{\R^n}
		\frac{\la y\ra^{-n-}}{|x-y|^{n-2}} \, dy
		\les \la x\ra^{2-n}.
	\end{align*}

\end{proof}	

The free resolvents $R_0^\pm(\lambda^2)$ appear in mulitple places
within the formula for $W_{s,2}$.  They are in fact convolution operators
whose kernel (for a given $n$) depends on $\lambda$, $|x-y|$, and the choice of sign.
Our starting point for handling~\eqref{eqn:Ws22} is to integrate with respect to $\lambda$ and apply
the following bound.

\begin{lemma} \label{lem:lambdaInt}
Let $R_0^\pm(\lambda^2,A)$ denote the convolution kernel of $R_0^\pm(\lambda^2)$
evaluated at a point with $|x-y| = A$.  For each $j \geq 0$,
\begin{equation}
	\int_0^\infty R_0^+(\lambda^2, A)\partial_B^j\big(R_0^+ - R_0^-\big)(\lambda^2,B)
	\lambda^{-1}\tilde\Phi(\lambda)\, d\lambda
	\les \begin{cases} \frac{1}{A^{n-2} \la A\ra^{n-2+j}} & \text{ if } A > 2B \\
	\frac{1}{A^{n-2} \la B\ra^{n-2+j}} & \text{ if } B > 2A \\
	\frac{1}{A^{n-2} \la A\ra \la A - B\ra^{n-3+j}}
	& \text{ if } A \approx B
	\end{cases} .
\end{equation}
This can be written more succinctly as
\begin{equation}
\int_0^\infty R_0^+(\lambda^2, A)\partial_B^j\big(R_0^+ - R_0^-\big)(\lambda^2,B)
	\lambda^{-1}\tilde\Phi(\lambda)\, d\lambda
	\les \frac{1}{A^{n-2}\la A+B\ra \la A-B\ra^{n-3+j}}.
\end{equation}
\end{lemma}

To handle some lower order terms in the expansion of $W_<$, we make use a related estimate.
\begin{corollary} \label{cor:lambdaInt}
Let $R_0^\pm(\lambda^2,A)$ denote the convolution kernel of $R_0^\pm(\lambda^2)$
evaluated at a point with $|x| = A$.  For each $j \geq 0$,
\begin{equation}
	\int_0^\infty R_0^+(\lambda^2, A)\partial_B^j\big(R_0^+ - R_0^-\big)(\lambda^2,B)
	\tilde\Phi(\lambda)\, d\lambda
	\les \begin{cases} \frac{1}{A^{n-2} \la A\ra^{n-1+j}} & \text{ if } A > 2B \\
	\frac{1}{A^{n-2} \la B\ra^{n-1+j}} & \text{ if } B > 2A \\
	\frac{1}{A^{n-2} \la A\ra \la A - B\ra^{n-2+j}}
	& \text{ if } A \approx B
	\end{cases} .
\end{equation}
This can be written more succinctly as
\begin{equation}
\int_0^\infty R_0^+(\lambda^2, A)\partial_B^j\big(R_0^+ - R_0^-\big)(\lambda^2,B)
	\tilde\Phi(\lambda)\, d\lambda
	\les \frac{1}{A^{n-2}\la A+B\ra \la A-B\ra^{n-2+j}}.
\end{equation}
\end{corollary}

\begin{rmk} 
The $j=0$ case is what appears in~\eqref{eqn:Ws22}.  
The $j=1$ and $j=2$ cases will be used to gain extra decay if $P_eV1 = 0$ and $P_eVx = 0$
respectively.
\end{rmk}

Based on Lemma~\ref{lem:lambdaInt}, we have
\begin{equation*}
	|K(x,y)|\les \iint_{\R^{2n}}\frac{|V\phi(z)| |V\phi(w)|\,dz \,dw}{|x-z|^{n-2} \la |x-z| +|y-w|\ra
	\la |x-z|- |y-w| \ra^{n-3}}.
\end{equation*}

If $V\phi(z)$ and $V\phi(w)$ decay rapidly enough, then this integral will be concentrated
primarily when $z$ and $w$ are small.

\begin{lemma}\label{lem:ring nocanc}

If $|V(z)|\les \la z\ra^{-(n-1)-}$,
we have the bound
\begin{align} \label{eqn:Kbound}
	|K(x,y)|&\les \iint_{\R^{2n}}\frac{|V\phi(z)| |V\phi(w)|\,dz\, dw}{|x-z|^{n-2} 
	\la |x-z| + |y-w|\ra
	\la |x-z| - |y-w| \ra^{n-3}}\nn \\
	&\les \frac{1}{\la x\ra^{n-2} \la |x| + |y| \ra 
	\la |x| - |y|\ra^{n-3}}.
\end{align}
\end{lemma}

We delay the proof of Lemmas~\ref{lem:lambdaInt} and~\ref{lem:ring nocanc},
and Corollary~\ref{cor:lambdaInt},
to Section~\ref{sec:appA}.  To show $L^p(\R^n)$ boundedness
of certain integral operators, we show that they have an admissible
kernel $K(x,y)$, that is
\begin{align}
	\sup_{x\in \R^n}\int_{\R^n} |K(x,y)|\, dy +
	\sup_{y\in \R^n}\int_{\R^n} |K(x,y)|\, dx <\infty.
\end{align}
It is well known that an operator with an admissible
kernel is bounded on $L^p(\R^n)$ for all $1\leq p\leq \infty$.
We now prove

\begin{prop}\label{prop:Kjkreg}

	If $|V(z)|\les \la z\ra^{-(n-1)-}$, then
	the operator with kernel $K(x,y)$ is bounded on $L^p(\R^n)$ for $ 1\leq p< \f n2$.

\end{prop}

\begin{proof}
In the region where $|x| > 2|y|$, \eqref{eqn:Kbound} shows that $|K(x,y)| \les \la x\ra^{4-2n}$.
We show that this is an admissible kernel.  The integral with respect to $x$ is uniformly bounded,
as $\la x\ra^{4-2n}$ is integrable provided $n > 4$.  The integral with respect to $y$ is
$\la x\ra^{4-2n} \int_{|y| < |x|/2} \,dy \les \la x\ra^{4-n}$, which is uniformly bounded so long as
$n \geq 4$.

In the region where $|x| \approx |y|$, \eqref{eqn:Kbound} shows that $|K(x,y)| \les \la y\ra^{1-n}
\la |x| - |y|\ra^{3-n}$.  This is also an admissible kernel.  By changing to spherical coordinates,
we have
\begin{equation*}
\int_{|x| \approx |y|} |K(x,y)|\,dx \les \la y\ra^{1-n} |y|^{n-1} \int_{|x|/2}^{2|x|} \frac{dr}{\la |x|-r\ra^{n-3}}
\les 1
\end{equation*}
since $n > 4$.

It is only the region where $|y| > 2|x|$ that creates restrictions on the $L^p$-boundedness of the
operator $K$.  Here~\eqref{eqn:Kbound} asserts that $|K(x,y)| \les \la x\ra^{2-n}\la y\ra^{2-n}$.
This is a bounded operator on $L^p(\R^n)$ if one can take the $L^{p'}$ norm in $y$ and then the
$L^p$ norm in $x$ with a finite result.  In this case
\begin{equation}  \label{eqn:Ky bigg x}
\int_{\R^n} \frac{1}{\la x\ra^{(n-2)p}} \Big(\int_{|y|>2|x|} \frac{dy}{\la y\ra^{(n-2)p'}}\Big)^{p-1}\,dx
\les \int_{\R^n} \frac{1}{\la x\ra^{np - 4p +n}} \,dx
\end{equation}
has a convergent inner integral provided $p' > \frac{n}{n-2}$, or in other words $p < n/2$.  The
second integral converges if $p(n-4) > 0$, which is always true for $n > 4$.
\end{proof}

\section{The cases of $P_eV1=0$ and $P_eVx=0$}\label{sec:canc}

As noted above, the operator kernel $K(x,y)$ is bounded on the entire range of
$L^p(\R^n)$, $1 \leq p \leq \infty$, if one considers only the region where $2|x| >|y|$. 
The restrictions on $p$ occur
on account of integrability concerns for large $y$ alone.
Specifically, in the region where $|y| > 2|x|$, \eqref{eqn:Kbound}
provides the bound $|K(x,y)| \les \la x\ra^{2- n}\la y\ra^{2- n}$,
which clearly decays at the rate $\la y\ra^{2-n}$ and no faster.

We  show that if the nullspace of $H$ satisfies orthogonality conditions
$P_eV1=0$ and $P_eVx = 0$, then $K(x,y)$ enjoys more rapid decay 
at infinity and
hence the operator is bounded on an expanded range of $L^p$ spaces.

\begin{prop} \label{prop:PV1}
Assume $P_eV1 = 0$ and $|V(z)|\les \la z\ra^{-n-1-}$. 
The operator with kernel $K(x,y)$ is bounded on $L^p(\R^n)$ for $ 1 \leq p<n$ whenever $n>4$.
\end{prop}

The original estimate~\eqref{eqn:Kbound} is already sufficient
to prove this in the regions where $2|x| > |y|$, 
and also when $|x|, |y| \leq 10$.
In fact, it is even possible to extract the desired decay in $\la y\ra$ from
much of the proof of Lemma~\ref{lem:A1} when $y$ is large.
More specifically,
\begin{lemma} \label{lem:largew}
Let $k \geq 0$.
If $N \geq 2n-3+k$ and $R \geq 0$ is fixed, then
\begin{equation} \label{eqn:largew}
	\int_{|w| > \frac{|y|}{2}}\frac{\la w\ra^{-N}}{\la R+|y-w|\ra \la R-|y-w|\ra^{n-3}} \, dw
	\les \frac{1}{\la R+|y| \ra \la R-|y| \ra^{n-3}\la y\ra^k}
\end{equation}
\end{lemma}
\begin{proof}
The estimates for~\eqref{eqn:ring1} and~\eqref{eqn:ring3} can be reproduced verbatim,
using $\alpha = 0$ and $\beta = n-3$,
the only change being that $N \geq 2n-3+k \geq n+\beta +k$ instead of $N \geq n+\beta$.  On the annulus where
$|y-w| \approx |y|$, excluding the ball $|w| < \frac{|y|}{2}$ yields the bound
\begin{align*}
\int_{\substack{|y-w| \approx |y| \\ |w| > \frac{|y|}{2}}} \frac{\la w\ra^{-N}}{\la R+|y-w|\ra\la R-|y-w|\ra^{n-3}}\,dw
&\les \la y\ra^{-N} \int_{\frac{|y|}{2}}^{2|y|} \frac{r^{n-1}}{\la R+r\ra \la R-r \ra^{n-3}}\,dr \\
&\les \frac{1}{\la y\ra^{N-n}\la R+|y|\ra} \max_{\frac{|y|}{2} < r < 2|y|} \la r-R\ra^{3-n}.
\end{align*}
If $R < \frac{|y|}{4}$  or $R > 4|y|$, then $\la r-R\ra \approx \la R-|y|\ra$  over the interval of integration,
and~\eqref{eqn:largew} is satisfied so long as $N \geq n+k$. 

When $R \approx |y|$, the maximum value of
$\la r-R\ra^{3-n}$ might be 1.  In that case $\la R - |y|\ra \les \la y \ra$, so the integral is bounded by
$\la y\ra^{n-1-N} \les \la y\ra^{2n-3-N} \la R+|y|\ra^{-1} \la R-|y|\ra^{3-n}$.
Then~\eqref{eqn:largew} is satisfied so long as $N \geq 2n-3+k$.
\end{proof}

\begin{proof}[Proof of Proposition~\ref{prop:PV1}]
By the discussion preceding Lemma~\ref{lem:largew},
we may assume that $|y| > 2|x|$ and $|y| > 10$,
and it suffices to demonstrate the bound
\begin{equation} \label{eqn:PV1}
|K(x,y)| \les \frac{1}{\la x\ra^{n-2}\la y\ra^{n-1}}
\end{equation}
in this region.

The cancellation condition $P_eV1 = 0$ allows us to rewrite the 
$K(x,y)$ integral in the following manner.
\begin{multline}	\label{eqn:KjkPV1}
K(x,y) = \iint_{\R^{2n}} \int_0^\infty V\phi(z) V\phi(w)
R_0^+(\lambda^2,|x-z|) \\
\big((R_0^+ - R_0^-)(\lambda^2, |y-w|) - (R_0^+ - R_0^-)(\lambda^2,|y|)\big)
\frac{\tilde\Phi(\lambda)}{\lambda} \,d\lambda \, dz\, dw.
\end{multline}
Subtracting the function $(R_0^+ - R_0^-)(\lambda^2,|y|)$, which is independent of $w$,
from the integrand does not affect the final value
because our standing assumption $P_eV1=0$ implies that
 $\int_{\R^n} V\phi(w)\,dw = 0$ for all eigenfunctions
 $\phi$.

For any function $F(\lambda, |y|)$ one can express
\begin{equation}\label{eqn:Taylor1}
F(\lambda, |y-w|) - F(\lambda,|y|) = \int_0^1 \partial_r F(\lambda, |y-sw|)
\frac{(-w) \cdot (y-sw)}{|y-sw|}\,ds.
\end{equation}
 Here we are interested in $F(\lambda, |y|)
= (R_0^+ - R_0^-)(\lambda^2,|y|) = (\frac{\lambda}{|y|})^{\frac{n-2}{2}}J_{\frac{n-2}{2}}(\lambda |y|)$,
whose radial derivatives are considered in the statement and proof of Lemma~\ref{lem:lambdaInt}.
Which side of the identity \eqref{eqn:Taylor1} we use is
decided based on the size of $|w|$ compared to $\f12 |y|$.
Accordingly, we divide the $w$ integal into two regions.
On the region where  $|w| < \frac12|y|$,
the contribution to $K(x,y)$ 
has the expression
\begin{multline} \label{eqn:K s}
\int_{|w| < \frac{|y|}{2}} \int_{\R^n} \int_0^\infty \int_0^1 
V\phi(z) V\phi(w)
R_0^+(\lambda^2,|x-z|) \partial_r\big((R_0^+ -R_0^-)(\lambda^2,|y-sw|)\big) \\
\frac{(-w)\cdot(y-sw)}{|y-sw|}
\frac{\tilde\Phi(\lambda)}{\lambda}\,ds\,d\lambda \, dz\, dw,
\end{multline}
where $\partial_r$ indicates the partial derivative with respect to the radial variable of
$R_0^\pm(\lambda^2,r)$.
Apply Fubini's Theorem and then Lemma~\ref{lem:lambdaInt} with $j=1$ to obtain the upper bound
\begin{equation*}
\int_0^1 \int_{|w| < \frac{|y|}{2}} \int_{\R^n} \frac{| V\phi(z)|\, |w V\phi(w)|}{|x-z|^{n-2}
	\la |x-z| + |y-sw|\ra \la|x-z| - |y-sw|\ra^{n-2}} \,dz\, dw\, ds.
\end{equation*}

By Lemma~\ref{lem:efn decay} and our assumption that $|V(z)| \les \la z\ra^{-(n+1)-}$,
we can control the decay of the numerator with $|V\phi(z)| \les \la z\ra^{-(2n-1)-}$
and $|w V\phi(w)| \les \la w\ra^{-(2n-2)-}$.  These are sufficient to apply
Lemma~\ref{lem:A1} in the $z$ variable, then Lemma~\ref{lem:B1} in the $w$ variable to obtain
\begin{align*}
|K(x,y)| \les \int_0^1 \int_{|w| < \frac{|y|}{2}}  &\frac{|V\phi(w)|}{\la x\ra^{n-2}\la |x| + |y-sw|\ra
 \la|x| - |y-sw|\ra^{n-2}} \,dw\, ds \\
&\les \int_0^1 \frac{1}{\la x\ra^{n-2}\la|x|+|y|\ra \la |x| - |y|\ra^{n-2}}\,ds 
\les \frac{1}{\la x\ra^{n-2} \la y\ra^{n-1}}
\end{align*}
when $|y| > 2|x|$, as desired.

For the portion of~\eqref{eqn:KjkPV1} where $|w| > \frac12|y|$, we treat the two
terms in the difference directly instead of 
rewriting as an integral using \eqref{eqn:Taylor1}.  For the term with $(R_0^+ - R_0^-)(\lambda^2,|y-w|)$,
direct applications of Lemmas~\ref{lem:lambdaInt} and~\ref{lem:largew} with $j=0$ and $k=1$ respectively,
followed by Lemma~\ref{lem:A1} with respect to the $z$ variable, shows that
\begin{align*}
\int_{\R^n}\int_{|w| > \frac{|y|}{2}} \int_0^\infty 
R_0^+(\lambda^2,|x-z|) V\phi(z) V\phi(w) (R_0^+ - R_0^-)(\lambda^2,|y-w|)
\frac{\tilde\Phi(\lambda)}{\lambda} \,d\lambda \, dw\, dz   \\
\les \int_{\R^n} \int_{|w|> \frac{|y|}{2}} \frac{|V\phi(z)|\, |V\phi(w)|}{
 |x-z|^{n-2}\la|x-z| + |y-w|\ra \la|x-z| -|y-w|\ra^{n-3}} \, dw\, dz \\
\les \frac{1}{\la x\ra^{n-2} \la |x| + |y| \ra \la |x| - |y|\ra^{n-3}\la y\ra} 
\les \frac{1}{\la x\ra^{n-2} \la y\ra^{n-1}} 
\end{align*}
within the region where $|y| > 2|x|$.

The estimate for the term with $(R_0^+ - R_0^-)(\lambda^2,|y|)$ is more straightforward.
\begin{align*}
\int_{\R^n}\int_{|w| > \frac{|y|}{2}} \int_0^\infty R_0^+(\lambda^2,|x-z|)V\phi(z) V\phi(w)
(R_0^+ - R_0^-)(\lambda^2,|y|)
\frac{\tilde\Phi(\lambda)}{\lambda} \,d\lambda \, dw \, dz \\
\les \int_{\R^n} \int_{|w|> \frac{|y|}{2}} \frac{|V\phi(z)|\, |V\phi(w)|}{
 |x-z|^{n-2}\la |x-z| + |y|\ra \la|x-z|-|y|\ra^{n-3}} \, dw\, dz \\
\les \frac{1}{\la x\ra^{n-2} \la |x|+ |y|\ra  \la |x| - |y|\ra^{n-3}\la y\ra} 
\les \frac{1}{\la x\ra^{n-2} \la y\ra^{n-1}}.
\end{align*}
We have used Lemma~\ref{lem:A1} for the integration in $z$, and 
the basic estimate $\int_{|w| > |y|/2} \la w\ra^{-N}\,dw \les \la y\ra^{n-N}$ in 
lieu of Lemma~\ref{lem:largew}.  We can now run
through the argument as in \eqref{eqn:Ky bigg x} to see
that the extra decay in $y$ allows the resulting integral
to converge provided $(n-1)p'>n$, which requires $p< n$.
\end{proof}

Finally, we show that with further cancellation, one can
extend to nearly the full range of $p$.  That is,

\begin{prop} \label{prop:PVx}
Assume $P_eV1 = 0$, $P_eVx=0$ and $|V(z)|\les \la z\ra^{-n-3-}$, then 
the operator with kernel $K(x,y)$ is bounded on $L^p(\R^n)$ for $ 1\leq p<\infty$ whenever $n>4$.
\end{prop}

\begin{proof}

We note that when $P_eV1, P_eVx=0$, we have the equality
\begin{multline}\label{eqn:2canc}
	\int_{\R^n}\int_0^\infty R_0^{+}(\lambda^2)(x,z) V(z)\phi(z)
	V(w) \phi(w) (R_0^+-R_0^-)(\lambda^2)(w,y)
	\frac{\tilde\Phi(\lambda)}{\lambda} \, d\lambda\,dw\\
	=\int_{\R^n} \int_0^\infty R_0^{+}(\lambda^2)(x,z) V\phi(z)
	V\phi(w) \Big[(R_0^+-R_0^-)(\lambda^2)(w,y) -F(\lambda,y) - G(\lambda,y)\frac{w\cdot y}{|y|}\Big]
	\frac{\tilde\Phi(\lambda)}{\lambda} \, d\lambda\,dw,
\end{multline}
for any functions $F(\lambda, y)$ and $G(\lambda, y)$.  
In place of \eqref{eqn:Taylor1}, we utilize the extra
level of cancellation to write
\begin{multline}\label{eqn:K2canc}
K(\lambda, |y-w|) - K(\lambda, |y|) + \partial_rK(\lambda,|y|) \frac{w \cdot y}{|y|}  \\
= \int_0^1 (1-s) \bigg[\partial_r^2 K(\lambda, |y-sw|)  \frac{(w \cdot (y-sw))^2}{|y-sw|^2}\\
+ \partial_r K(\lambda, |y-sw|)\Big(\frac{|w|^2}{|y-sw|} - \frac{(w \cdot (y-sw))^2}{|y-sw|^3}\Big)\bigg]\, ds.
\end{multline}
The formula above suggests that we choose
$F(\lambda, y)=(R_0^+-R_0^-)(\lambda^2,|y|)$ and
$G(\lambda,y)=\partial_r (R_0^+-R_0^-)(\lambda^2,|y|)$
in \eqref{eqn:2canc} respectively.

As in the proof of Proposition~\ref{prop:PV1}, whenever
$|w|>|y|/2$, we can use the decay of $V\phi(w)$ to limit its contribution to $K(x,y)$
to a term of size $\la x\ra^{2-n}\la y\ra^{-n}$ when $|y| > 2|x|$.  

If, on the other hand, $|w|<|y|/2$,
there are new terms to bound of the  form
\begin{multline*}
\int_{|w|<\frac{|y|}{2}}\int_{\R^n} \int_0^\infty \int_0^1 
V\phi(z) V\phi(w)
R_0^+(\lambda^2,|x-z|)
\partial_r^j\big((R_0^+ -R_0^-)(\lambda^2,|y-sw|)\big) \\
(1-s)\Gamma_j(s,w,y)
\frac{\tilde\Phi(\lambda)}{\lambda}\,ds\,d\lambda \, dz\,dw
\end{multline*}
with $j=1,2$ and
$\Gamma_j(s,w,y)$ denoting
$$
	\Gamma_1 = \Big(\frac{|w|^2}{|y-sw|} - \frac{(w \cdot (y-sw))^2}{|y-sw|^3}\Big), \qquad
	\textrm{and} \qquad  \Gamma_2 = \frac{(w \cdot (y-sw))^2}{|y-sw|^2}.
$$  
Within the range $|w| < |y|/2$ and $0 \leq s \leq 1$, these factors observe the bounds 
$|\Gamma_1(s,w,y)| \les |y|^{-1}|w|^2$ and $|\Gamma_2(s,w,y)| \leq |w|^2$.
The calculation continues in the same manner as in Proposition~\ref{prop:PV1}, first
using Lemma~\ref{lem:lambdaInt} with $j=1,2$, then
Lemma~\ref{lem:A1} in the $z$ integral and Lemma~\ref{lem:B1} (with $\alpha = 2-j$) 
in the $w$ integral.  
For both values of $j$ we arrive at the bound $\la x\ra^{2-n} \la y\ra^{-n}$ when $|y|>2|x|$.

Put together with the previous claim, this implies that in fact
\begin{equation} \label{eqn:Kboundcanc}
|K(x,y)| \les \frac{1}{\la x\ra^{n-2}\la y\ra^{n}} \ \text{when} \ |y| > 2|x|. 
\end{equation} 
The estimates from~\eqref{eqn:Kbound} sill hold when $|x| \approx |y|$ and $|x| > 2|y|$.
In the region where $|y| > 2|x|$ we can imitate the calculation in~\eqref{eqn:Ky bigg x}
and find convergent integrals so long as $p' > 1$.
The operator with kernel $K(x,y)$ is therefore 
bounded on $L^p(\R^n)$ for all $p \in [1,\infty)$.

The extra growth of $|w|^2$ in the size of $\Gamma_j(s,w,y)$ dictates the amount of decay we need on the potential.
We must use
Lemma~\ref{lem:B1} with $\beta=n-3+j$, which
requires $|w|^2|V\phi(w)| \les \la w\ra^{1-2n}$, from which 
Lemma~\ref{lem:efn decay} shows that $|V(w)|\les \la w\ra^{-n-3-}$ is needed.
\end{proof}

\begin{rmk}
It appears to be possible to make an analogous cancellation argument to improve decay with respect to $x$
in the region $|x| > 2|y|$.  Given the existing strength of~\eqref{eqn:Kbound} here, the benefits of doing so
are unclear.  When $|x| \approx |y|$, cancellation only leads to improvement in the exponent of $\la |x| - |y|\ra$,
which does not affect the integrability properties of $K(x,y)$ in a meaningful way.
\end{rmk}

We can now prove the main Theorem.

\begin{proof}[Proof of Theorem~\ref{thm:main}]

We first prove the desired results for $n>3$ odd.
In this case, one has the expansion
$W_<=\Phi(H)(1-(W_{r,0}+W_r+W_{s,1}+W_{s,2}))\Phi(H_0)$,
where
\begin{equation}
\begin{aligned}
	W_{r,0}&=\frac{1}{\pi i} \int_0^\infty R_0^+(\lambda^2)
	V(R_0^+(\lambda^2)-R_0^-(\lambda^2))\lambda\, d\lambda\\
	W_{r}&=\frac{1}{\pi i} \int_0^\infty R_0^+(\lambda^2)
	VA_0(\lambda)(R_0^+(\lambda^2)-R_0^-(\lambda^2))
	\tilde \Phi(\lambda)\lambda\, d\lambda\\
	W_{s,1}&=\frac{1}{\pi i} \int_0^\infty R_0^+(\lambda^2)
	VA_{-1}(R_0^+(\lambda^2)-R_0^-(\lambda^2))
	\tilde \Phi(\lambda)\, d\lambda\\
	W_{s,2}&=\frac{1}{\pi i} \int_0^\infty R_0^+(\lambda^2)
	VP_eV(R_0^+(\lambda^2)-R_0^-(\lambda^2))
	\tilde \Phi(\lambda)\lambda^{-1}\, d\lambda \label{eqn:Ws2}
\end{aligned}
\end{equation}
Here the subscripts
$r$ denotes regular and $s$ denotes singular terms.
In \cite{Yaj}, Yajima shows that the first two `regular' terms are
bounded on $L^p$ for $1\leq p\leq \infty$.  In particular,
it is shown that $\|W_{r,0}u\|_p\les \|\mathcal F (\la \cdot \ra^{\sigma}V)\|_{L^{n_*}}\|u\|_p$ for $\sigma>\frac{1}{n_*}$.

	Under the assumptions on the decay of $V$ and
	\eqref{Vfour}, it was shown in \cite{Yaj} that
	$W_>$, $\Phi(H)(1-(W_{r,0}+W_r))\Phi(H_0)$
	are all bounded on $L^p$ for $1\leq p\leq \infty$.
	Further, $W_{s,1}=0$ if $n>5$.  We note that, if $n=5$
	and $P_eV1=0$ that $W_{s,1}=0$, Section 3.2.2 in \cite{Yaj}.  Thus, we need only bound
	$W_{s,2}$ for all $n>4$.

	The kernels of $\Phi(H)$ and $\Phi(H_0)$ are bounded
	by $C_N \la x-y \ra^{-N}$ for each $N=1,2,\dots$, see
	Lemma~2.2 of \cite{YajWkp3}.  Following \eqref{eqn:Ws2}, 
	\eqref{eqn:Ws22}, $W_{s,2}$ is
	bounded on $L^p$ exactly when the operators
	$K^{jk}$ are.  Proposition~\ref{prop:Kjkreg} proves the
	first claim for all $n>5$,
	Proposition~\ref{prop:PV1}
	proves the second, while Proposition~\ref{prop:PVx}
	proves the third.
	
	We need to make a few adjustments to our approach when
	$n=5$ if there is no cancellation.  The operator 
	$W_{s,1}$ may be rewritten as
	\begin{align*}
		\frac{1}{\pi i} \int_0^\infty R_0^+(\lambda^2)
		[V(P_0V) \otimes V(P_0V)](R_0^+(\lambda^2)-R_0^-(\lambda^2))
		\tilde \Phi(\lambda)\, d\lambda,
	\end{align*}
	where $P_0V=\sum \phi_j \la V, \phi_j \ra$ is a function
	with the same decay properties as an eigenfunction
	in Lemma~\ref{lem:efn decay}.  The $L^p$ boundedness
	follows by using Corollary~\ref{cor:lambdaInt}
	and modifying Proposition~\ref{prop:Kjkreg} accordingly.

	The proof for $n$ even is quite similar.  We note that
	\cite[Section 2.2]{YajNew} allows us to express
	$W_<=\Phi(H)(W_{r}+W_{\log}+W_{s,2})\Phi(H_0)$.
	When $n>10$, it is shown there that $W_{\log}$ vanishes
	and $W_r$ is bounded on $L^p(\R^n)$ for all $1\leq p\leq \infty$.    These results,
	plus our Proposition~\ref{prop:Kjkreg}, establish the first
	claim of Theorem~\ref{thm:main}.  Noting equation (2.19) in \cite{YajNew}, if
	$n > 6$ and $P_eV1=0$, the operator $W_{\log}$ vanishes, so
	Proposition~\ref{prop:PV1}
	proves the second claim, while 
	Proposition~\ref{prop:PVx}
	proves the third.  

	When $n=8,10$,
	one also needs to control the contribution 
	of $W_{\log}$ in order to prove the first claim of Theorem~\ref{thm:main}.
	When $n=6$ it is not clear that $W_{\log}$ vanishes under any of the
	given cancellation conditions.  We show that it is bounded on $L^p(\R^6)$ for all
	$1 \leq p < \infty$, which is sufficient to complete the proof for all three claims
	in Theorem~\ref{thm:main}.
	The additional arguments require some modification of the main
	techniques used and are presented separately in Section~\ref{sec:8,10}
	below.

\end{proof}

\begin{rmk}\label{PVx2 rmk}

	We note that the endpoint
	$p=\infty$ is not covered in our analysis when
	$P_eV1,P_eVx=0$.  If, we add in the additional
	assumption $P_eVx^2=0$ and $|V(z)|\les \la z\ra^{-n-5-}$, 
	we can use the techniques
	above to show that the wave operators are bounded
	on $L^p(\R^n)$ for the full range of $1\leq p\leq \infty$.  Here the assumption $P_eVx^2=0$ means
	that $\int_{\R^n} P_2(x)V(x)\phi(x)\, dx=0$ for
	any quadratic monomial $P_2$.
	We leave the details to the reader.

\end{rmk}

\section{Integral Estimates}\label{sec:appA}

\subsection{Proof of Lemma~\ref{lem:lambdaInt}}
We first recall the inequality claimed in Lemma~\ref{lem:lambdaInt}, namely
\begin{equation*}
	\int_0^\infty R_0^+(\lambda^2, A)\partial_B^j\big(R_0^+ - R_0^-\big)(\lambda^2,B)
	\lambda^{-1}\tilde\Phi(\lambda)\, d\lambda
	\les \begin{cases} \frac{1}{A^{n-2} \la A\ra^{n-2+j}} & \text{ if } A > 2B \\
	\frac{1}{A^{n-2} \la B\ra^{n-2+j}} & \text{ if } B > 2A \\
	\frac{1}{A^{n-2} \la A\ra \la A - B\ra^{n-3+j}}
	& \text{ if } A \approx B
	\end{cases} .
\end{equation*}

The proof is a lengthy exercise in integration by parts.  The following elementary bound
will be invoked repeatedly.
\begin{lemma} \label{lem:IBP}
Suppose there exists $\beta > -1$ and $M > \beta+1$
such that $|F^{(k)}(\lambda)| \les \lambda^{\beta - k}$ for all $0 \leq k \leq M$.
Then given a smooth cutoff function $\tilde\Phi$,
\begin{equation} \label{eqn:IBP}
\Big|\int_0^\infty e^{i\rho \lambda}F(\lambda) \tilde\Phi(\lambda)\,d\lambda \Big| 
\les \la \rho\ra^{-\beta-1}.
\end{equation}
If $F$ is further assumed to be smooth and supported in the annulus $L \les \lambda \les 1$
for some $L > \rho^{-1}>0$, then
\begin{equation} \label{eqn:IBP2}
\Big|\int_0^\infty e^{i\rho \lambda}F(\lambda) \tilde\Phi(\lambda)\,d\lambda \Big| \les \la \rho\ra^{-M}L^{\beta+1-M}.
\end{equation}
\end{lemma}

\begin{proof}
It is immediately true that~\eqref{eqn:IBP} is uniformly bounded.  Assume $\rho \gtrsim 1$.
Then 
\begin{equation*}
\bigg|\int_0^{\rho^{-1}} e^{i\rho\lambda}F(\lambda) \tilde\Phi(\lambda)\,d\lambda \bigg| \les 
\int_0^{\rho^{-1}} \lambda^{\beta}\,d\lambda \les \rho^{-\beta-1}.
\end{equation*}
By repeated integration by parts,
\begin{equation*}
\bigg|\int_{\rho^{-1}}^\infty e^{i\rho\lambda}F(\lambda) \tilde\Phi(\lambda) \,d\lambda \bigg| \les
\sum_{k=1}^M \rho^{-k}|F^{(k-1)}(\rho^{-1})| 
+\rho^{-M} \int_{\rho^{-1}}^\infty \bigg| \big({\textstyle \frac{d}{d\lambda}}\big)^M (F\tilde\Phi)(\lambda)
\bigg|\,d\lambda
\les \rho^{-\beta-1}.
\end{equation*}
We also use the fact that $|\tilde \Phi^{(k)}(\lambda)|\les \lambda^{-k}$ in the last inequality.

In the second case, the integral over $0 \leq \lambda \leq \rho^{-1}$ is empty, and no
boundary terms are created by the integration by parts.  Then
\begin{equation*}
\Big|\int_0^\infty e^{i\rho \lambda}F(\lambda) \tilde\Phi(\lambda)\,d\lambda \Big| 
\les \rho^{-M} \int_L^\infty  \big({\textstyle \frac{d}{d\lambda}}\big)^M (F\tilde\Phi)(\lambda)\,d\lambda
\les \rho^{-M} L^{\beta+1-M}.
\end{equation*}
\end{proof}

\begin{proof}[Proof of Lemma~\ref{lem:lambdaInt}]
For complex $z \in {\mathbb C} \setminus [0,\infty)$, the free resolvent $(H_0-z)^{-1}$ is a convolution operator whose
kernel may be expressed in terms of Hankel functions.
\begin{align}\label{Hankel}
	R_0(z)(x,y)=\frac{i}{4} \bigg(\frac{z^{1/2}}{2\pi |x-y|}\bigg)^{\frac{n}{2}-1} H_{\frac{n}{2}-1}^{(1)}(z^{1/2}|x-y|).
\end{align}
Here $H_{\frac{n}{2}-1}^{(1)}(\cdot)$ is the Hankel function of the first kind and $z^{1/2}$ is defined to
take values in the upper halfplane.   The Hankel functions of half-integer order, which arise when $n$ is odd,
can be expressed more simply as the product of an exponential function and a polynomial.

Formula~\eqref{Hankel} yields a set of asymptotic descriptions of $R_0^\pm(\lambda^2, A)$,
based on the properties of Hankel and Bessel functions in~\cite{AS}.
We will use the fact that
$R_0^\pm(\lambda^2, A) \approx A^{2-n}$ when $\lambda A \les 1$, and 
$R_0^\pm(\lambda^2, A) = e^{\pm i\lambda A} A^{2-n} \Psi_{\frac{n-3}{2}}(\lambda A)$
when $\lambda A \gtrsim 1$.  In odd dimensions, the function $\Psi_{\frac{n-3}{2}}$ is actually a polynomial of order
$\frac{n-3}{2}$.  In even dimensions it is a function that asymptotically behaves like $(\,\cdot\, )^{(n-3)/2}$
and whose $k^{th}$ derivative behaves like $(\,\cdot\,)^{(n-3-2k)/2}$.

For the difference of resolvents, we have the low-energy description
$R_0^+(\lambda^2,B) - R_0^-(\lambda^2,B) = 
(\frac{\lambda}{B})^{\frac{n-2}{2}} J_{\frac{n-2}{2}}(\lambda B)
\approx \lambda^{n-2}$.  This is
a real-analytic function of $\lambda$ and $B$ so it has well behaved derivatives of all orders.
We concisely describe the resolvent kernels for $n \geq 3$ and their
differences as follows:

\begin{align}
R_0^\pm(\lambda^2, A) &= \frac{1}{A^{n-2}}\Omega (\lambda A) + \frac{e^{\pm i\lambda A}}{A^{n-2}}
\Psi_{\frac{n-3}{2}}(\lambda A), \\
\label{eqn:asymptotics1}
R_0^+(\lambda^2, B) - R_0^-(\lambda^2, B) &= \lambda^{n-2}\Omega(\lambda B)
+ \frac{e^{i\lambda B}}{B^{n-2}} \Psi_{\frac{n-3}{2}}(\lambda B)
+ \frac{e^{-i\lambda B}}{B^{n-2}} \Psi_{\frac{n-3}{2}}(\lambda B), \\
\label{eqn:asymptotics}
\partial_B^j\big(R_0^+(\lambda^2,B) - R_0^-(\lambda^2,B)\big) &= \lambda^{n-2+j} \Omega(\lambda B)
  + \frac{e^{i\lambda B}}{B^{n-2}} \lambda^j \Psi_{\frac{n-3}{2}}(\lambda B) \\
& \qquad+ \frac{e^{-i\lambda B}}{B^{n-2}} \lambda^j \Psi_{\frac{n-3}{2}}(\lambda B). \notag
\end{align}
where $\Omega$ is a bounded compactly supported function that is smooth everywhere except possibly at zero,
and each $\Psi_{\frac{n-3}{2}}$ is a smooth function
supported outside the unit interval with the asymptotic behavior specified above.
In odd dimensions, the polynomial $\Psi_{\frac{n-3}{2}}(\lambda B)$ in~\eqref{eqn:asymptotics}
may include negative powers of order up to $(\lambda B)^{-j}$.  Since the support of $\Psi_{\frac{n-3}{2}}$ is
bounded away from zero, the presence of negative powers is of little consequence.

With some abuse of notation, we use $\Omega$ and $\Psi_{\frac{n-3}{2}}$ to describe several distinct functions.
It should be understood that all we need from these functions are the properties that this
notation yields.  As a relevant example,
we say that $\frac{d}{d\lambda} \Psi_{\frac{n-3}{2}}(\lambda A)
= A \Psi_{\frac{n-5}{2}}(\lambda A) = \lambda^{-1} \Psi_{\frac{n-3}{2}}(\lambda A)$ and
$\frac{d}{dA} \Psi_{\frac{n-3}{2}}(\lambda A) = \lambda \Psi_{\frac{n-5}{2}}(\lambda A)
= A^{-1} \Psi_{\frac{n-3}{2}}(\lambda A)$.  The new functions $\Psi_{\frac{n-3}{2}}$ retain the general 
properties of the original, but with different constants and/or coefficients.

First case: $A > 2B$. We split the integral into four pieces:
\begin{equation} \label{eqn:intAlarge}
\begin{aligned}
\int_0^\infty 
\bigg(\frac{\lambda^{n-3+j}}{A^{n-2}}\Omega(\lambda A) \Omega(\lambda B) 
&+
\frac{e^{i\lambda A}}{A^{n-2}}\lambda^{n-3+j}\Psi_{\frac{n-3}{2}}(\lambda A) \Omega(\lambda B) \\
&+ \frac{e^{i\lambda(A- B)}}{A^{n-2}B^{n-2}} \lambda^{j-1}
	\Psi_{\frac{n-3}{2}}(\lambda A)\Psi_{\frac{n-3}{2}}(\lambda B) \\
&+ \frac{e^{i\lambda(A+ B)}}{A^{n-2}B^{n-2}} \lambda^{j-1}
	\Psi_{\frac{n-3}{2}}(\lambda A) \Psi_{\frac{n-3}{2}}(\lambda B)
\bigg)\tilde\Phi(\lambda)\,d\lambda.
\end{aligned}
\end{equation}
The first term is supported where $\lambda \les \la A\ra^{-1}$ and is bounded pointwise by $A^{2-n}\lambda^{n-3+j}$,
so its integral is clearly dominated by $A^{2-n}\la A\ra^{2-n-j}$.  All other terms are empty unless $A \gtrsim 1$.

The second term of~\eqref{eqn:intAlarge} is supported in an interval with $\frac{1}{A} \les \lambda \les \frac{1}{B}$,
and goes to zero smoothly at both endpoints.  We can write
\begin{equation*}
\lambda^{n-3+j}\Psi_{\frac{n-3}{2}}(\lambda A) \Omega(\lambda B)
= A^{\frac{n-3}{2}}F(\lambda).
\end{equation*}
Observe that $F(\lambda)$ satisfies the conditions of Lemma~\ref{lem:IBP} with $\beta = \frac{3n-9}{2}+j$, and
constants uniform in $A$ and $B$.  If any derivatives fall on $\Omega(\lambda B)$, the resulting factor of 
$B$ is dominated by $\lambda^{-1}$ inside the support of $F(\lambda)$.
Thus
\begin{equation*}
\int_0^\infty \frac{e^{i\lambda A}}{A^{n-2}}\lambda^{n-3+j}\Psi_{\frac{n-3}{2}}(\lambda A) \Omega(\lambda B)
\tilde\Phi(\lambda)\,d\lambda \les \frac{1}{A^{\frac{n-1}{2}}\la A\ra^{\frac{3n-7}{2}+j}}
\les \frac{1}{A^{n-2}\la A\ra^{n-2+j}}.
\end{equation*}

The third and fourth terms of~\eqref{eqn:intAlarge} are quite similar, so we will consider only the third term here.
This is supported in an interval with $\frac{1}{B} \les \lambda \les 1$ and goes smoothly to zero at both endpoints.
We can write
\begin{equation*}
\lambda^{j-1}\Psi_{\frac{n-3}{2}}(\lambda A) \Psi_{\frac{n-3}{2}}(\lambda B)
= A^{\frac{n-3}{2}}B^{\frac{n-3}{2}}F(\lambda),
\end{equation*}
where $F(\lambda)$ satisfies the stronger conditions of Lemma~\ref{lem:IBP} 
with $\beta = n-4+j$ and $L = B^{-1}$.  Thus
\begin{align*}
\Big| \int_0^\infty \frac{e^{i\lambda(A- B)}}{A^{n-2}B^{n-2}} \lambda^{j-1}
\Psi_{\frac{n-3}{2}}(\lambda A) \Psi_{\frac{n-3}{2}}(\lambda B)\tilde\Phi(\lambda)\,d\lambda \Big|
&\les \frac{1}{A^{\frac{n-1}{2}}B^{\frac{3n-7}{2}+j-M}\la A- B\ra^M} \\
&\les \frac{1}{A^{n-2}\la A\ra^{n-2+j}},
\end{align*}
using the assumptions $A > 2B$ (so that $\la A-B \ra \approx \la A \ra$), $A \gtrsim 1$, 
and choosing $M \geq \frac{3n-7}{2}+j$.  The fourth term
is less delicate, as we have identical bounds with $\la A+B\ra$ in place of
$\la A-B\ra$.

Second case: $B > 2A$.
This time the integral splits into five terms:
\begin{multline} \label{eqn:intBlarge}
\int_0^\infty 
\bigg(\frac{\lambda^{n-3+j}}{A^{n-2}}\Omega(\lambda A) \Omega(\lambda B) 
+
\frac{e^{i\lambda B}}{(AB)^{n-2}}\lambda^{j-1}\Psi_{\frac{n-3}{2}}(\lambda B) \Omega(\lambda A) \\
+ \frac{e^{-i\lambda B}}{(AB)^{n-2}}\lambda^{j-1}\Psi_{\frac{n-3}{2}}(\lambda B) \Omega(\lambda A) 
+ \frac{e^{i\lambda(A- B)}}{(AB)^{n-2}} \lambda^{j-1}
	\Psi_{\frac{n-3}{2}}(\lambda A) \Psi_{\frac{n-3}{2}}(\lambda B) \\
+ \frac{e^{i\lambda(A+ B)}}{(AB)^{n-2}} \lambda^{j-1}
	\Psi_{\frac{n-3}{2}}(\lambda A) \Psi_{\frac{n-3}{2}}(\lambda B)
\bigg)\tilde\Phi(\lambda)\,d\lambda.
\end{multline}

The first term is supported where $\lambda \les \la B\ra^{-1}$ and is bounded pointwise by
$A^{2-n}\lambda^{n-3+j}$, so its integral is dominated by $A^{2-n}\la B\ra^{2-n-j}$ as desired.
All other terms are empty unless $B \gtrsim 1$.

The second and third terms are similar, so we will only consider the second term here.
We can write
\begin{equation*}
\lambda^{j-1} \Psi_{\frac{n-3}{2}}(\lambda B) \Omega(\lambda A) = B^{\frac{n-3}{2}}F(\lambda),
\end{equation*}
where $F(\lambda)$ satisfies the conditions of Lemma~\ref{lem:IBP} with $\beta = \frac{n-5}{2}+j$.
Note that $A \les \lambda^{-1}$ inside the support of $\Omega(\lambda A)$.  As a consequence,
\begin{equation*}
\Big|\int_0^\infty \frac{e^{i\lambda B}}{(AB)^{n-2}}
\lambda^{j-1} \Psi_{\frac{n-3}{2}}(\lambda B) \Omega(\lambda A) \tilde\Phi(\lambda) \,d\lambda \Big|
\les \frac{1}{A^{n-2}B^{\frac{n-1}{2}}\la B\ra^{\frac{n-3}{2}+j}} \les \frac{1}{A^{n-2}\la B\ra^{n-2+j}}
\end{equation*}
since $B \gtrsim 1$.

The fourth and fifth terms are also very similar, and we consider only the fourth term here.
The integrand is supported in an interval with $\frac{1}{A} \les \lambda \les 1$ and goes
smoothly to zero at both endpoints.  We can write
\begin{equation*}
\lambda^{j-1} \Psi_{\frac{n-3}{2}}(\lambda A) \Psi_{\frac{n-3}{2}}(\lambda B) 
= (AB)^{\frac{n-3}{2}} F(\lambda),
\end{equation*}
where $F(\lambda)$ satisfies the stronger conditions of Lemma~\ref{lem:IBP} with $\beta = n-4+j$
and $L = \frac{1}{A}$.  It follows from~\eqref{eqn:IBP2} that
\begin{equation*}
\Big| \int_0^\infty \frac{e^{i\lambda(A- B)}}{(AB)^{n-2}} \lambda^{j-1}
	\Psi_{\frac{n-3}{2}}(\lambda A) \Psi_{\frac{n-3}{2}}(\lambda B) \tilde\Phi(\lambda) \Big|
\les \frac{1}{A^{\frac{3n-7}{2}+j-M}B^{\frac{n-1}{2}}\la B \ra^{M}} 
\les \frac{1}{A^{n-2}\la B\ra^{n-2+j}}
\end{equation*}
using the assumption that $B > 2A$ and choosing $M \geq n-3+j > \frac{n-3}{2}+j$. 

Third case: $A \approx B$.  If one is overly cautious, there might be as many as six terms to consider
in the integral.
\begin{multline} \label{eqn:intAB}
\int_0^\infty 
\bigg(\frac{\lambda^{n-3+j}}{A^{n-2}} \Omega(\lambda A) \Omega(\lambda B) 
+
\frac{e^{i\lambda A}}{A^{n-2}}\lambda^{n-3+j}\Psi_{\frac{n-3}{2}}(\lambda A) \Omega(\lambda B) \\
+
\frac{e^{i\lambda B}}{(AB)^{n-2}}\lambda^{j-1}\Psi_{\frac{n-3}{2}}(\lambda B) \Omega(\lambda A) 
+ \frac{e^{-i\lambda B}}{(AB)^{n-2}}\lambda^{j-1} \Psi_{\frac{n-3}{2}}(\lambda B) \Omega(\lambda A) \\ 
+ \frac{e^{i\lambda(A- B)}}{A^{n-2}B^{n-2}} \lambda^{j-1}
	\Psi_{\frac{n-3}{2}}(\lambda A) \Psi_{\frac{n-3}{2}}(\lambda B) \\
+ \frac{e^{i\lambda(A+ B)}}{A^{n-2}B^{n-2}} \lambda^{j-1}
	\Psi_{\frac{n-3}{2}}(\lambda A) \Psi_{\frac{n-3}{2}}(\lambda B)
\bigg)\tilde\Phi(\lambda)\,d\lambda.
\end{multline}

The first term is supported where $\lambda \les \la A\ra^{-1}$ and is bounded pointwise by
$A^{2-n}\lambda^{n-3+j}$, so its integral is dominated by $A^{2-n}\la A\ra^{2-n-j}$.  Recalling that $A\approx B$,
all other terms
are empty unless $A,B \gtrsim 1$.

Note that the next three terms all contain products with both $\Psi_{\frac{n-3}{2}}$ and $\Omega$.
Since $A \approx B$, these products are supported in an interval $\lambda \approx A^{-1}$ and have size
bounded by $1$.  Consequently each of the integral terms is bounded by $A^{4-2n-j}$ and is zero for small $A$.

The fifth term plays a very significant role.  We can write
\begin{equation*}
\lambda^{j-1} \Psi_{\frac{n-3}{2}}(\lambda A) \Psi_{\frac{n-3}{2}}(\lambda B) = A^{n-3} F(\lambda),
\end{equation*}
where $F(\lambda)$ satisfies the conditions of Lemma~\ref{lem:IBP} with $\beta = n-4+j$.  Consequently,
\begin{equation*}
\Big| \int_0^\infty \frac{e^{i\lambda(A- B)}}{A^{n-2}B^{n-2}} \lambda^{j-1}
	\Psi_{\frac{n-3}{2}}(\lambda A)\Psi_{\frac{n-3}{2}}(\lambda B) \tilde\Phi(\lambda\,d\lambda \Big|
\les \frac{1}{A^{n-1}\la A- B\ra^{n-3+j}},
\end{equation*}
and is also zero for small $A$.

The final term is treated in much the same way, however the phase function $e^{i\lambda(A+B)}$
leads to a stronger bound of $A^{1-n}\la A + B\ra^{3-n-j} \approx A^{4-2n-j}$, since $A\gtrsim 1$.

\end{proof}

The proof of Corollary~\ref{cor:lambdaInt} follows by the
observation that multiplication by $\lambda$ is essentially
the same as taking a radial derivative $\partial_B \big(R_0^+(\lambda^2,B) - R_0^-(\lambda^2,B)\big)$, see
\eqref{eqn:asymptotics}.

\subsection{Proof of Lemma~\ref{lem:ring nocanc}}

The main inequality in Lemma~\ref{lem:ring nocanc} is as follows.
\begin{equation*}
\iint_{\R^{2n}}\frac{|V\phi(z)| |V\phi(w)|\,dz dw}{|x-z|^{n-2} 
	\la |x-z| + |y-w|\ra
	\la |x-z| - |y-w| \ra^{n-3}}
	\les \frac{1}{\la x\ra^{n-2} \la |x| + |y| \ra 
	\la |x| - |y|\ra^{n-3}}.
\end{equation*}

The integrals will be set up in spherical coordinates, with radial variables $|x-z|$ and $|y-w|$, 
so the first step is the following estimate of integrals along shells.

\begin{lemma}\label{lem:shells}

	If $N>n-1$, then 
	\begin{align*}
		\bigg|\int_{|x-z|=r}\la z\ra^{-N} dz \bigg|
		\les \left\{\begin{array}{ll}
		r^{n-1} \la x\ra^{-N} & r<\frac{1}{2}|x|\\
		|x|^{n-1}\la x\ra^{1-n} \la |x|-r\ra^{n-1-N} 
		& \frac{1}{2}|x|\leq r\leq 2|x|\\
		r^{n-1}\la r\ra^{-N} & r>2|x|
		\end{array}
		\right.
	\end{align*}

\end{lemma}

\begin{proof}

The cases $r < \frac12 |x|$ and $r > 2|x|$ are both trivial, as $|z|$ is comparable to $|x|$ and $|r|$ respectively.
In addition, if $r \approx |x| \les 1$, then $\la z\ra^{-N} \les 1$ over a sphere of radius approximately $|x|$.
This leaves only the case where $r \approx |x|$ and both are relatively large.

We write the integral in spherical coordinates, with $\theta$ representing the angle between $z-x$ and $-x$.
Then $|z|^2 = |x|^2 + r^2 - 2|x|r\cos\theta$. We can estimate
\begin{equation*}
\int_{|x-z|=r} \la z\ra^{-N}dz = 
C_n \int_0^\pi \frac{r^{n-1}\sin^{n-2}\theta}{\la(|x|-r)^2 +2|x|r(1-\cos{\theta})\ra^{N/2}}d\theta.
\end{equation*}
One may replace $\sin\theta$ with $\theta$, and $(1-\cos\theta)$ with $\theta^2$ for the purposes
of establishing an upper bound.

On the ``cap" where $\theta < ||x|-r|/\sqrt{|x|r}$, the integral may be estimated by
\begin{equation*}
\int_0^\frac{||x|-r|}{\sqrt{|x|r}} \frac{r^{n-1}\theta^{n-2}}{\la|x|-r \ra^{N}}d\theta
\les  ||x|-r|^{n-1}\la|x|-r\ra^{-N},
\end{equation*}
keeping in mind that $r \approx |x|$.  On the remaining interval, we have
\begin{equation*}
\int_\frac{||x|-r|}{\sqrt{|x|r}}^\pi \frac{r^{n-1}\theta^{n-2}}{\la|x|\theta\ra^{N}}d\theta
\approx \int_{||x|-r|}^{\pi|x|} \frac{\alpha^{n-2}\,d\alpha}{\la \alpha\ra^N}
\les \la|x|-r\ra^{n-1-N},
\end{equation*}
using the substitution $\alpha = |x|\theta$.
\end{proof}

With this bound in hand, we consider first the $z$ integral
of Lemma~\ref{lem:ring nocanc}. 

\begin{lemma} \label{lem:A1}

	Let $\beta \geq 1$ and $0 \leq \alpha < n-1$.  If $N\geq n+\beta$, 
	then for each fixed constant $R\geq 0$, we have the bound
	$$
	\int_{\R^n}\frac{\la z\ra^{-N}}{|x-z|^\alpha\la |x-z| + R\ra \la |x-z|-R\ra^\beta} \, dz
	\les \frac{1}{\la x\ra^\alpha \la |x| + R\ra \la R-|x| \ra^\beta}.
	$$

\end{lemma}

\begin{proof}

With $R\geq 0$ a 
fixed constant, we integrate $\R^n$ in shells of radius
$r=|x-z|$.
\begin{align*}
	\int_{\R^n}&\frac{\la z\ra^{-N}}{|x-z|^\alpha \la|x-z|+R\ra \la |x-z|-R\ra^\beta} \, dz	\\
	& \qquad = \int_0^\infty \frac{1}{r^\alpha \la r+R\ra \la r-R\ra^\beta}
	\int_{|z-x|=r} \la z\ra^{-N} \, dz \, dr.
\end{align*}
By the bounds in Lemma~\ref{lem:shells}, we bound the
above integral with a sum of three integrals,
\begin{align}
	\int_{\R^n} &\frac{\la z\ra^{-N}}{|x-z|^\alpha
	\la|x-z|+R\ra \la |x-z|-R\ra^\beta} \, dz\nn \\
	& \qquad \les\frac{1}{\la x\ra^{N}}\int_0^{|x|/2}\frac{r^{n-1-\alpha}}{\la r+R\ra \la r-R\ra^\beta}\, dr\label{eqn:ring1}\\
	& \qquad + \frac{|x|^{n-1-\alpha}}{\la x\ra^{n-1}} \int_{|x|/2}^{2|x|} 
	\frac{1}{\la r+R\ra \la r-R\ra^\beta \la r-|x|\ra^{N-n+1}} \, dr\label{eqn:ring2}\\
	& \qquad +\int_{2|x|}^\infty \frac{r^{n-1-\alpha}}{\la r+R\ra \la r-R
	\ra^\beta \la r \ra^{N}} \, dr.\label{eqn:ring3}
\end{align}
We estimate each piece individually. For \eqref{eqn:ring1},
we consider two cases.  First, if $R<\f34 |x|$, we cannot
use the decay of $\la r-R\ra^{-\beta}$ effectively, and instead note that
\begin{align*}
	\eqref{eqn:ring1}&\les \frac{1}{\la x\ra^{N}}\int_0^{|x|/2}\frac{r^{n-1-\alpha}}{\la r\ra}\, dr 
\les \frac{|x|^{n-1-\alpha}}{\la x\ra^{N}}\les \frac{1}
	{\la x\ra^{N-n+\alpha - \beta}\la |x|+R\ra \la R-|x|\ra^\beta}.
\end{align*}
The integral bound used here requires $\alpha < n-1$, 
and we used that $|x| \approx |x|+R$ and $|\,|x|-R|\leq |x|$ in the last step.
The desired bound follows, provided that $N \geq n+\beta$.

In the second case, when $R>\f34 |x|$,we have
$|R-r|>\f13 R$ and $r+R \approx R \approx |x|+R$, so that
\begin{align*}
	\eqref{eqn:ring1}&\les \frac{|x|^{n-\alpha}}{\la x\ra^{N} \la R \ra^{\beta+1}}
	\leq \frac{1}{\la x\ra^{N-n+\alpha} \la |x|+R\ra \la R -|x|\ra^\beta}.
\end{align*}
The last inequality is valid because $|R-|x|\,|\leq R$ in this case.

Now we consider the region on which $r\approx |x|$ in 
\eqref{eqn:ring2}.  On this region, we have
\begin{align*}
	\eqref{eqn:ring2}\les \frac{|x|^{n-1-\alpha}}
	{\la x\ra^{n-1}\la |x|+R\ra}
	\int_{r\approx |x|} \frac{1}{\la r-R \ra^{\beta} \la r-|x|
	\ra^{N-n+1}} \, dr.
\end{align*}
We now extend the integral to $\R$ and apply
the simple bound, which asserts that
\begin{align}\label{bracket decay}
    \int_{\R^m} \langle y \rangle^{-\gamma} \langle x-y\rangle^{-\mu} \, dy
    \lesssim \langle x \rangle^{-\min(\gamma, \mu)},					
\end{align}
for all choices $0 < \mu, \gamma$ with $\max(\gamma, \mu) > m$.
For the integral under consideration, one chooses
$m=1$, $\gamma=\beta$ and $\mu=N-n+1$ to obtain
\begin{align*}
	\eqref{eqn:ring2}\les \frac{1}
	{\la x\ra^\alpha \la |x|+R\ra \la R-|x| \ra^\beta},
\end{align*}
provided $N -n+1 > \beta$.

Finally, we note consider the final region of integration,
\eqref{eqn:ring3}.
We wish to control
\begin{align*}
	\int_{2|x|}^\infty \frac{r^{n-1-\alpha}}{\la r+R\ra \la r-R
	\ra^\beta \la r \ra^{N}} \, dr.
\end{align*}
We consider two cases.  First, if $R<\f32 |x|$, then since $r > 2|x|$ this
implies that $|R-|x|\,| \leq |r-R|$ and also $r+R \approx r$. 
Thus, in this case we have
\begin{align*}
	\eqref{eqn:ring3}\les \frac{1}{\la R-|x|\ra^\beta}
	\int_{2|x|}^\infty \la r\ra^{n-2-\alpha-N}\, dr
	&\les \frac{1}{\la R-|x|\ra^\beta \la x \ra^{N-n+1+\alpha}} \\
	&\les \frac{1}{\la x\ra^{N-n+\alpha}\la |x|+R\ra\la R-|x|\ra^\beta},
\end{align*}
provided $N>n-1-\alpha$ so that the integral converges.  We used $|x|+R \leq \frac52|x|$ in the last step.
The desired bound follows if $N\geq n$.

The second case is when $R>\f32 |x|$, in which case we have
$|R-|x|\,| \approx R \approx |x|+R$.  We break up the region of integration
further into three subregions.  First, if $2|x|\leq r\leq \f R2$, we note that $|R-r|\geq \f R2 \approx |R-|x|\,|$, 
permitting the bounds
\begin{align*}
	\int_{2|x|}^{\f R2}\frac{r^{n-1-\alpha}}{\la r+R\ra \la r-R\ra^\beta \la r \ra^N}\, dr
	& \les \frac{1}{\la R\ra \la R-|x|\ra^\beta} \int_{2|x|}^\infty \frac{r^{n-1-\alpha}}{\la r \ra^{N}}\, dr \\
	& \les \frac{1}{\la|x|+R\ra \la R-|x|\ra^\beta \la x\ra^{N-n+\alpha}}.
\end{align*}

The next region we consider is when $\f R2 \leq r\leq 2R$.
Since $r\approx R$, we gain nothing from $\la R-r\ra^{-\beta}$ and treat it as a constant.
Instead, we note that $|R-|x|\,| \approx R$, and thus we bound with
\begin{align*}
	\int_{\f R2}^{2R} \frac{r^{n-1-\alpha}}{\la r+R\ra \la R-r\ra^\beta
	\la r\ra^{N}}\, dr &\les \frac{R^{n-\alpha}}{\la R-|x|\ra^\beta \la R\ra^{N+1-\beta}} \\
	&\les \frac{1}{\la R-|x|\ra^\beta} \frac{1}{\la R
	\ra^{N-n +1+\alpha-\beta}}.
\end{align*}
Then, we note that $|x|\les R$, in order to bound with
$$
	\frac{1}{\la x\ra^{N-n+\alpha-\beta}\la |x|+R\ra \la R-|x|\ra^{\beta}}.
$$
This yields the desired bound, provided
$N \geq n +\beta$.

We now consider the last case in which $2R\leq r$.
We note that  $|R-r|\approx r \approx r+R$ to bound
\begin{align*}
	\int_{2R}^\infty \frac{r^{n-1-\alpha}}{\la r+R\ra \la R-r \ra^\beta
	\la r\ra ^N}\, dr\les 
	\int_{2R}^\infty \frac{1}{\la r \ra^{N-n+\alpha +\beta + 2}}\, dr
	\les \frac{1}{\la R \ra^{N-n + 1 +\alpha+\beta}}.
\end{align*}
We now note that $|R-|x|\, |\approx R$ and $R> \f32 |x|$
to bound with
\begin{align*}
	\frac{1}{\la x\ra^{N-n+\alpha}\la |x|+R\ra \la R-|x| \ra^\beta}.
\end{align*}
This yields the desired bound provided $N\geq n$
\end{proof}

We now conclude with the proof of Lemma~\ref{lem:ring nocanc}.

\begin{proof}[Proof of Lemma~\ref{lem:ring nocanc}]

Recall that we wish to bound the following,
\begin{align*}
	\iint_{\R^{2n}}\frac{|V\phi (z)| |V\phi (w)|}{|x-z|^{n-2}\la|x-z|+|w-y|\ra
	\la |x-z| - |y-w| \ra^{n-3}} \, dz\, dw.
\end{align*}

We consider the $z$ integral first.
Using Lemma~\ref{lem:efn decay} and the assumed decay $|V(z)|\les \la z\ra^{-(n-1)-}$,
it follows that $|V\phi(z)| \les \la z\ra^{-N}$ for some $N > 2n-3$. 
Fix the constant $R=|w-y|\geq 0$ and apply
Lemma~\ref{lem:A1} (with parameters $\alpha = n-2$ and $\beta = n-3$),
to conclude that $K(x,y)$ is bounded by the integral
\begin{align*}
	\frac{1}{\la x\ra^{n-2} }
	\int_{\R^n} \frac{\la w\ra^{-N}}{\la |x| + |w-y|\ra \la |w-y|-|x| \ra^{n-3}}\, dw.
\end{align*}
We again apply Lemma~\ref{lem:A1}, this time in $w$ with
$R=|x|\geq 0$ and $\alpha = 0$, to bound with
\begin{align*}
	\frac{1}{\la x\ra^{n-2} \la |x|+|y|\ra \la |y|-|x| \ra^{n-3}},
\end{align*}
as desired.
\end{proof}

\begin{lemma}\label{lem:B1}

	Suppose $|y|>10$  and $0<s\leq 1$.   Let $\alpha> 0$ and $\beta \geq 1$.
	If $N\geq n+\beta$, then for each fixed constant $R\geq 0$, we have the bound
	\begin{equation}  \label{eqn:B1}
	\int_{|w| < \frac{|y|}{2}}\frac{\la w\ra^{-N}}{|y-sw|^\alpha\la |y-sw|+R\ra \la |y-sw|-R\ra^\beta} \, dw
	\les \frac{1}{\la y\ra^\alpha \la |y|+R\ra \la R-|y| \ra^\beta}.
	\end{equation}

\end{lemma}

\begin{rmk}
The extra assumptions of large $y$ and relatively small $w$ allow us to remove
the upper restriction on the size of $\alpha$.  This is very important because
we need $\alpha = n-1$ and $\alpha=n$ for the cases
of $P_eV1=0$ and $P_eVx = 0$ respectively.
\end{rmk}

\begin{proof}

We prove this in a similar manner to Lemma~\ref{lem:A1},
taking care to show that the new parameter $s$ is essentially harmless if $|y|$ is large.  We begin by
decomposing the integral into shells,
\begin{multline*}
	\int_{\R^n}\frac{\la w\ra^{-N}}{|y-sw|^\alpha\la |y-sw|+R\ra\la |y-sw|-R\ra^\beta} \, dw \\
	= \int_0^\infty \frac{1}
	{|sr|^\alpha \la sr+R\ra \la sr-R \ra^\beta} \int_{|y-sw|=sr}
	\la w\ra ^{-N}\, dw \, dr\\
	=\int_0^\infty \frac{1}
	{|sr|^\alpha \la sr + R\ra\la sr-R \ra^\beta} \int_{|w-\f ys|=r}
	\la w\ra ^{-N}\, dw \, dr.
\end{multline*}
As before, we can use Lemma~\ref{lem:shells} to break this
into three pieces.  The relevant bound is
\begin{align*}
	\bigg|\int_{|w-\f ys|=r}\la w\ra^{-N} dw \bigg|
	\les \left\{\begin{array}{ll}
	r^{n-1} \la y/s\ra^{-N} & r<\frac{|y|}{2s}\\
	\big(\frac{|y|}{s}\big)^{n-1}\la \frac{y}{s}\ra^{1-n} \la r-\frac{|y|}{s}\ra^{n-1-N} 
	& \frac{|y|}{2s}\leq r\leq \frac{2|y|}{s}\\
	r^{n-1}\la r\ra^{-N} & r>\frac{2|y|}{s}
	\end{array}
	\right..
\end{align*}

However we are not integrating over all of $\R^n$.  In fact,
the ball $|w| < \frac{|y|}{2}$ consists of points where
$sr = |y-sw|$ lies between $\frac12|y|$ and $\frac32|y|$
and is bounded away from zero. 
Thus only the shell where $r \approx \frac{|y|}{s}$ makes any
contribution to the integral.  Furthermore, $\frac{|y|}{s} \gtrsim 1$,
so the factors $(\frac{|y|}{s})^{n-1} \la\frac{y}{s}\ra^{1-n}$
neatly cancel each other.

We use that 
$sr\approx |y|$ to see that the contribution of the integral can be bounded by
\begin{align*}
	\int_{\frac{|y|}{2s}}^{2\frac{|y|}{s}}
	&\frac{1}
	{|sr|^\alpha \la sr+ R\ra\la sr-R \ra^\beta \la r - \frac{|y|}{s}\ra^{N-n+1}}\, dr  \\
	& \qquad \les \frac{1}{|y|^\alpha \la |y|+R\ra}
	\int_{\frac{|y|}{2s}}^{2\frac{|y|}{s}}
	\frac{1}
	{\la sr-R \ra^\beta \la r - \frac{|y|}{s}\ra^{N-n+1}}\, dr.
\end{align*}
We make the change of variables $q=rs$, to see
\begin{align*}
	\frac{1}{|y|^\alpha \la |y|+R\ra}
	\int_{\frac{|y|}{2}}^{2|y|}
	\frac{1}
	{\la q-R \ra^\beta \la \frac{q-|y|}{s} \ra^{N-n+1}}\, \frac{dq}{s}.	
\end{align*}
We note that if $||y|-R|\les 1$, then~\eqref{eqn:B1} is satisfied so long as the
integral above is bounded uniformly in $s$.
When $q\approx r$, we
replace $\la q - R\ra^{-\beta}$ by a constant and observe that
remaining expression is a portion of $\int_\R \la r - \frac{|y|}{s}\ra^{n-1-N}\,dr$,
 which is integrable and independent of $s$ provided $N > n$.

There are two cases to consider when $||y|-R||\gtrsim 1$.
The first case is when $|q-|y||<\f12 |R-|y||$.  Here, we
use $|q-R|=|q-|y|+|y|-R|\geq ||y|-R|-|q-|y||>\f12 ||y|-R|$.  Thus, we bound with
\begin{multline*}
	\frac{1}{|y|^\alpha\la |y|+R\ra \la |y|-R \ra^\beta}
	\int_{\frac{|y|}{2}}^{2|y|}
	\frac{1}
	{\la \frac{q-|y|}{s} \ra^{N-n+1}}\, \frac{dq}{s}	\\
	 \qquad \les
	\frac{1}{|y|^\alpha\la |y|+R\ra \la |y|-R \ra^\beta}
	\int_{\R}
	\frac{1}
	{\la r- \frac{|y|}{s} \ra^{N-n+1}}\, dr	
	 \les  \frac{1}{|y|^\alpha \la|y|+R\ra \la |y|-R \ra^\beta}
\end{multline*}
as desired.

In the second case when $|q-|y||>\f12 |R-|y||\gtrsim 1$,
so that $\frac{q-|y|}{s}\gtrsim 1$.  In this case,
we bound with
\begin{multline*}
	\frac{1}{|y|^\alpha\la|y|+R\ra } \int_{|q-|y||>\f12 |R-|y||}
	\frac{dq}{\la q-R \ra^\beta \big(\frac{q-|y|}{s}
	\big)^{N-n+1}s}\\
	\les \frac{s^{N-n}}{|y|^\alpha \la|y|+R\ra} \int_{|q-|y||>\f12 |R-|y||}
	\frac{dq}{\la q-R \ra^\beta \la q-|y|
	\ra^{N-n+1}}\\
	\les \frac{s^{N-n}}{|y|^\alpha\la|y|+R\ra} \int_{\R }
	\frac{dq}{\la q-R \ra^\beta \la q-|y|
	\ra^{N-n+1}} 
	\les \frac{s^{N-n}}{|y|^\alpha \la|y|+R\ra \la R-|y| \ra^\beta}.
\end{multline*}
Where the last inequality follows from \eqref{bracket decay} provided $N>n-1+\beta$.

\end{proof}

\section{Completing the cases $n=6, 8,10$}\label{sec:8,10}

In dimensions $n=8,10$, we still have the task of controlling
the contribution of $W_{\log}$, which vanishes only if $P_eV1 = 0$.
From the expansions for
$(1+R_0^+(\lambda^2)V)^{-1}$ in Theorem~2.3 of \cite{YajNew},
this term takes the form
\begin{align}
	W_{\log}=\frac{1}{\pi i} \int_0^\infty R_0^+(\lambda^2)
	[V(P_eV) \otimes V(P_eV)](R_0^+(\lambda^2)-R_0^-(\lambda^2))
	\tilde \Phi(\lambda)\lambda^{j-1}(\log \lambda)^\ell\, d\lambda
\end{align}
To control the contribution of these terms, we can
adjust the techniques used previously to control the
operator $W_{s,2}$.  We use the following modification
of Lemma~\ref{lem:IBP}.

\begin{lemma} \label{lem:IBP2}
Suppose there exists $\beta > -1$ and $M > \beta+1$
such that $|F^{(k)}(\lambda)| \les \lambda^{\beta - k}|\log \lambda|^\ell$ for all $0 \leq k \leq M$.
Then given a smooth cutoff function $\tilde\Phi$,
\begin{equation} \label{eqn:IBPv2}
\Big|\int_0^\infty e^{i\rho \lambda}F(\lambda) \tilde\Phi(\lambda)\,d\lambda \Big| 
\les \la \rho\ra^{-\beta-1} \la \log \la\rho\ra \ra^\ell.
\end{equation}
If $F$ is further assumed to be smooth and supported in the annulus $L \les \lambda \les 1$
for some $L > \rho^{-1}>0$, then
\begin{equation} \label{eqn:IBP2v2}
\Big|\int_0^\infty e^{i\rho \lambda}F(\lambda) \tilde\Phi(\lambda)\,d\lambda \Big| \les \la \rho\ra^{-M}L^{\beta+1-M} \la \log L \ra^{\ell}.
\end{equation}
\end{lemma}

\begin{proof}
The proof follows as in the proof of Lemma~\ref{lem:IBP},
with a few simple modifications.  When $0<\lambda<\rho^{-1}$, it follows that
$$
	\bigg|\int_0^{\rho^{-1}}\lambda^\beta (\log \lambda)^\ell \, d\lambda \bigg| \les \rho^{-\beta} 
	\bigg| \int_0^{\rho^{-1}}(\log \lambda)^\ell\, d\lambda 
	\bigg| \les \rho^{-\beta-1} \la \log \rho \ra^\ell .
$$
On the other hand, if $\lambda >\rho^{-1}$, we note that
on the support of $\tilde \Phi$, that 
$|\log \lambda|\les |\log \rho|$.  The second claim follows
similarly with $L$ replacing $\rho^{-1}$.

Finally, for $\rho < 1$, integrability of $F$ ensures that the left side of~\eqref{eqn:IBPv2}
is bounded by a constant independent of $\rho$.  Inequality~\eqref{eqn:IBP2v2} is
vacuously true for $\rho \ll 1$ because then $ L > \rho^{-1} \gg 1$ does not allow
$F(\lambda)$ to be non-zero in the support of the integral under the assumptions of the Lemma.

\end{proof}

Similar to the proof of Lemma~\ref{lem:lambdaInt}, we can prove

\begin{lemma}

Let $R_0^\pm(\lambda^2,A)$ denote the convolution kernel of $R_0^\pm(\lambda^2)$
evaluated at a point with $|x| = A$.  For each $j \geq 0$,
\begin{multline}
	\int_0^\infty R_0^+(\lambda^2, A)\big(R_0^+ - R_0^-\big)(\lambda^2,B)
	\lambda^{j-1}(\log \lambda)^\ell
	\tilde\Phi(\lambda)\, d\lambda\\
	\les \left\{\begin{array}{ll} \frac{\la \log \la A \ra \ra^\ell}{A^{n-2} \la A\ra^{n-2+j}} & \text{ if } A > 2B \\
	\frac{\la \log \la B\ra \ra^\ell }{A^{n-2} \la B\ra^{n-2+j}} & \text{ if } B > 2A \\
	\frac{\la \log \la A - B \ra  \ra^\ell}{A^{n-2} \la A\ra \la A - B\ra^{n-3+j}}
	& \text{ if } A \approx B
	\end{array}
	\right.
\end{multline}
This can be written more succinctly as
\begin{equation}
\int_0^\infty R_0^+(\lambda^2, A)(R_0^+ - R_0^-\big)(\lambda^2,B)
	\lambda^{j-1}(\log \lambda)^\ell \tilde\Phi(\lambda)\, d\lambda
	\les \frac{\la \log \la A-B\ra\ra^\ell}{A^{n-2}\la A+B\ra \la A-B\ra^{n-3+j}}.
\end{equation}

\end{lemma}

The proof follows by simply following the proof of
Lemma~\ref{lem:lambdaInt}
using Lemma~\ref{lem:IBP2} in place of Lemma~\ref{lem:IBP}.
In the regime where $A \approx B$ most of the terms in the
decomposition~\eqref{eqn:intAB} can be bounded
by $A^{2-n}\la A\ra^{2-n-j} \la \log \la A\ra\ra^\ell$, except for
the fifth term which is bounded instead by $A^{1-n}\la A-B\ra^{3-n-j}
\la \log \la A-B\ra\ra^\ell$.  Since $A$ and $B$ are positive numbers,
$1\leq \la A-B\ra \leq \la A\ra$.  Moreover, so long as $\ell \leq n-3+j$ 
the function 
$G(t) = \frac{ \la \log t\ra^\ell}{t^{n-3+j}}$ is decreasing for all
$t > 1$.  The end result is that the upper bound for the fifth term
is always an effective upper bound for the entire sum.

When $n = 8,10$, the form of $W_{\log}$ expressed in~\cite{YajNew}
consists of a single term with the values $j = n-4$, $\ell = 1$,
which certainly satisfies $ \ell \leq n-3+j$.

We can then repeat the  arguments in Lemma~\ref{lem:ring nocanc}
with the lazy bound $\la \log\la A-B\ra \ra \les \la A-B\ra^{0+}$.
This bounds the kernel of $W_{\log}$ by the quantity
\begin{multline*} 
	\iint_{\R^{2n}}\frac{|V\tilde\phi(z)| |V\tilde\phi(w)| \la \log \la |x-z|- |y-w|\ra\ra^{\ell}}{|x-z|^{n-2} 
	\la |x-z| + |y-w|\ra
	\la |x-z| - |y-w| \ra^{n-3+j}}\,dz dw\\
	\les 
	\iint_{\R^{2n}}\frac{|V\tilde\phi(z)| |V\tilde\phi(w)|  }{|x-z|^{n-2} 
	\la |x-z| + |y-w|\ra
	\la |x-z| - |y-w| \ra^{n-3+j-}}\,dz dw	
	\\ \les \frac{1}
	{\la x\ra^{n-2} \la |x| + |y| \ra  \la |x| - |y|\ra^{n-3+j-}}.
\end{multline*}
where $\tilde \phi$ satisfies the same decay properties
as an eigenfunction $\phi$.  Since $j = n-4 > 2$, this is
sufficient to make $W_{\log}$ an admissible kernel
that is bounded on $L^p(\R^n)$ for all $1 \leq p \leq \infty$.

If $n =6$ the structure of $W_{\log}$ is more complicated.
Assuming that $|V(x)| \leq C\la x\ra^{-\beta}$ for $\beta > 10$,
Theorem 2.3 of~\cite{YajNew} provides the expression
\begin{align}
	W_{\log}= \sum_{j,\ell =1}^2 \sum_{a,b=1}^{2d} \int_0^\infty R_0^+(\lambda^2)
	[\varphi_a \otimes \psi_b ](R_0^+(\lambda^2)-R_0^-(\lambda^2))
	\tilde \Phi(\lambda)\lambda^{2j-1}(\log \lambda)^\ell\, d\lambda,
\end{align}
where $d$ is the (finite) dimension of the zero energy eigenspace, and each
$\varphi_a, \psi_b$ belongs to $\la x\ra^{-\beta+3+}H^2(\R^6)$.
Dependence of $\varphi_a$ and $\psi_b$ on the parameters $j$ and $\ell$
is suppressed in the notation above.  Each term in the sum yields an integral
kernel that is bounded by
\begin{multline*} 
	\iint_{\R^{2n}}\frac{|\varphi_a(z)| |\psi_b(w)| \la \log \la |x-z|- |y-w|\ra\ra^{\ell}}{|x-z|^{4} 
	\la |x-z| + |y-w|\ra
	\la |x-z| - |y-w| \ra^{3+2j}}\,dz\, dw\\
	\les 
	\iint_{\R^{2n}}\frac{|\varphi_a(z)| |\psi_b(w)|  }{|x-z|^{4} 
	\la |x-z| + |y-w|\ra
	\la |x-z| - |y-w| \ra^{3+2j-}}\,dz\, dw.
\end{multline*}

It is not possible to invoke Lemma~\ref{lem:A1} directly because we lack a pointwise
bound for functions $\varphi_a$, $\psi_b$.  However they do belong to the space
$\la x\ra^{-\beta+3+}L^6(\R^6)$ by Sobolev embedding.  Then by H\"older's inequality
we may write
\begin{align*}
\int_{\R^6}&\frac{|\varphi_a(z)|}{|x-z|^4\la |x-z| + R\ra \la |x-z|-R\ra^{3+2j-}} \, dz \\
& \quad \leq 
\| \la z\ra^{\beta-3-} \varphi_a \|_6 \ \lnorm \frac{1}{\la z\ra^{1/6} \la |x-z| + R\ra^{1/6}}\rnorm_\infty \\
& \hspace{.5in} \times
\bigg(\int_{\R^6} \frac{\la z\ra^{-\frac65(\beta-\frac16-3-)}}
{(|x-z|^4 \la |x-z|-R\ra^{3+2j-})^{6/5}\la |x-z|+R\ra} \, dz \bigg)^{5/6} \\
& \quad\les \| \la z\ra^{\beta-3-} \varphi_a \|_6 \ \la |x| + R\ra ^{-\frac16}
\bigg(\frac{1}{\la x\ra^4 \la |x| + R\ra^{5/6} \la|x| - R\ra^{3+2j -}} \bigg) \\
& \quad \les \frac{1}{\la x\ra^4 \la |x|+R\ra \la |x|-R\ra^{3+2j-}}.
\end{align*}
The $L^\infty$ bound is observed via the inequality $\la z\ra \la |x-z| +R\ra
\gtrsim \la|z| + |x-z| + R\ra \geq \la |x|+ R\ra$.
In order to apply Lemma~\ref{lem:A1} to the last integral, we need the singularity
$|x-z|^{-24/5}$ to have an exponent less than $6-1$,
and for $\frac65(\beta - \frac{19}{6} -) \geq 6 + \frac65(3+2j-)$.
The first condition is true, and the second is satisfied provided $\beta \geq 11+ \frac16 + 2j$.
Allowing $\beta > 16$ suffices in all cases.

The integral involving $\psi_b(w)$ is handled in an identical manner.
After summing over $a$, $b$, and $\ell$, it follows that the kernel for $W_{\log}$ is bounded by
\begin{equation*}
\sum_{j=1}^2 \frac{1}{\la x\ra^4 \la |x|+|y|\ra \la |x|-|y|\ra^{3+2j-}}
\end{equation*}
If $j=2$ this is an admissible kernel whose operator is bounded on $L^p(\R^6)$
for all $1\leq p \leq \infty$.  If $j=1$, the bound of $\frac{1}{\la x\ra^4\la y\ra^{-6-}}$
for large $y$ just fails to be integrable. One can follow the proof of Proposition~\ref{prop:Kjkreg}
to determine that this integral operator is still bounded on $L^p(\R^6)$ for all
$1 \leq p < \infty$, missing only the $p=\infty$ endpoint.

We believe that careful analysis of the operators $D_{jk}^{(i)}$ 
derived in~\cite{FY} would show that $W_{\log}$ vanishes when $P_eV1 = 0$ in dimension six, just as
it does when $n = 8, 10$.  Even without this claim, however, the bound for
$W_{\log}$ given here suffices to complete the proof of Theorem~\ref{thm:main}
in all of its cases.

\end{document}